\documentclass[letterpaper,11pt,reqno]{amsart}

\usepackage{amssymb}
\usepackage{amsmath}
\usepackage{amsthm}
\usepackage{amsfonts}
\usepackage{bbm}
\usepackage{color}

\usepackage{hyperref}
\hypersetup{pdfstartview={FitH}}

\def\R{\mathbb{R}}
\def\Z{\mathbb{Z}}

\renewcommand{\L}{\mathrm{L}}

\newtheorem{theorem}{Theorem}
\newtheorem{lemma}[theorem]{Lemma}

\newtheorem{corollary}[theorem]{Corollary}

\theoremstyle{definition}
\newtheorem{remark}[theorem]{Remark}

\DeclareFontFamily{U}{mathx}{\hyphenchar\font45}
\DeclareFontShape{U}{mathx}{m}{n}{
<5> <6> <7> <8> <9> <10>
<10.95> <12> <14.4> <17.28> <20.74> <24.88>
mathx10
}{}
\DeclareSymbolFont{mathx}{U}{mathx}{m}{n}
\DeclareFontSubstitution{U}{mathx}{m}{n}
\DeclareMathAccent{\widecheck}{0}{mathx}{"71}

\numberwithin{equation}{section}

\begin{document}
\title{Boundedness of   some multi-parameter fiber-wise  multiplier operators}

\author{Fr{\'e}d{\'e}ric Bernicot}
\address{Fr{\'e}d{\'e}ric Bernicot, CNRS - Universit{\'e} de Nantes, Laboratoire Jean Leray, 2 rue de la Houssini{\'e}re, 44322 Nantes cedex 3, France}
\email{frederic.bernicot@univ-nantes.fr}

\author{Polona Durcik}
\address{Polona Durcik, California Institute of Technology, 1200 E California Blvd, Pasadena CA 91125,~USA}
\email{durcik@caltech.edu}

\date{\today}


\begin{abstract}
We prove  $\L^p$ estimates for various multi-parameter bi- and trilinear  operators with symbols  acting on fibers of the  two-dimensional functions.
 In particular, this yields estimates for the general bi-parameter form of the twisted paraproduct   
 studied in \cite{vk:tp}.
\end{abstract}

\maketitle

\section{Introduction}
The classical Coifman-Meyer theorem  \cite{cm1,cm2}  is concerned with    bilinear   operators    of the form
\begin{align*} 
T_m(F_1,F_2)(x)= \int_{\R^{2n}} \widehat{F_1}(\xi) \widehat{F_2}(\eta) m(\xi,\eta) e^{2\pi i x\cdot (\xi+\eta)} d\xi d\eta,
\end{align*}
defined for test functions $F_1,F_2:\R^n\rightarrow \mathbb{C}$ and $m$ a bounded function on $\R^{2n}$. 
The Coifman-Meyer theorem  states that if   $m$, in addition, 
  satisfies 
\begin{align}
\label{cmsym}
|\partial_\xi^\alpha  \partial_\eta^\beta m(\xi,\eta)| \leq C   |(\xi,\eta)|^{-|\alpha|-|\beta|}
\end{align}
for all multi-indices $\alpha,\beta\in \mathbb{N}^n_0$ up to a sufficently large finite order and all $(\xi,\eta)\neq 0$, with   $0\leq C <\infty$, then the     operator $T_m$ maps $\L^{p_1}\times \L^{p_2}$ to $\L^{p_3}$ whenever $1< p_1,p_2\leq \infty$, $1/2<p_3<\infty$, and $1/p_1+1/p_2=1/p_3$. A notable application of this result is to the fractional Leibniz rule by Kato and Ponce \cite{KP}, which has further applications to nonlinear PDE; see for instance the work by Christ and Weinstein \cite{CW}.
 
Multi-parameter variants of the Coifman-Meyer theory arise by   considering multilinear operators with symbols behaving as tensor  products of   symbols \eqref{cmsym}.
A simple  bi-parameter example  can be obtained by considering the operator $T_m$ with  
\begin{align}
\label{mult-m-product}
m(\xi,\eta) =m_1(\xi)m_2(\eta),
\end{align}
 where $m_1$ and $m_2$ are smooth away from the origin and satisfy the analogous  estimates as in \eqref{cmsym}. This case immediately  splits into a pointwise product of two linear Calder{\'o}n-Zygmund operators. It can be observed that the symbol $m$ in \eqref{mult-m-product} satisfies  
 \begin{align}
 \label{sym-m-prod}
 |\partial_\xi^\alpha  \partial_\eta^\beta m(\xi,\eta)| \leq C  |\xi|^{-|\alpha|} |\eta|^{-|\beta|}.
 \end{align}
 In contrast to this example, a major early contribution to the theory of bi-parameter operators    by Grafakos and Kalton \cite{GK} states that the condition \eqref{sym-m-prod}   is in general not sufficient for the $\L^p$ boundedness of  $T_m$.

Further developments of the multi-parameter theory were driven by interest in obtaining various fractional Leibniz-type rules such as in the works by Muscalu, Pipher, Tao, Thiele \cite{mptt:bi,mptt:multi}. In particular,  these  papers   show bounds  for the      operators $T_m$ with symbols satisfying  
\begin{align}
\label{mptt-sym}
|\partial_{(\xi_1,\eta_1)}^\alpha  \partial_{(\xi_2,\eta_2)}^\beta m(\xi,\eta)| \leq C  |(\xi_1,\eta_1)|^{-|\alpha|} |(\xi_2,\eta_2)|^{-|\beta|}	,
\end{align}
where  $\xi=(\xi_1,\xi_2),\,\eta=(\eta_1,\eta_2)\in \R^{n}\times \R^n$.
More general operators with   so-called flag singularities were  studied by Muscalu \cite{Mus:flag1, Mus:flag2}. 
Some  recent works in the area include the one  by Muscalu and Zhai \cite{Zhai-Muscalu}, who  investigate a certain trilinear operator which falls under the class of singular Brascamp-Lieb integrals with non-H\"older scaling, and some related   flag paraproducts studied by Lu, Pipher, and Zhang \cite{lu-pipher-zhang}. 
 
In the aforementioned  papers, the  symbols in  question 
generalize  products of  Coifman-Meyer symbols \eqref{cmsym}, each symbol acting on 
one or several fibers of the input functions. For instance,   the fiber-wise action in  \eqref{mptt-sym}    means that the first  factor on the right-hand side of \eqref{mptt-sym} concerns only   $F_1(\cdot,\xi_2)$ and $F_2(\cdot,\eta_2)$, while the second term concerns only the complementary  fibers $F_1(\xi_1,\cdot)$ and $F_2(\eta_1,\cdot)$.  
 Several recent developments in the theory of  singular integrals  include   the study of multilinear  operators
with    symbols acting fiber-wise on the input functions but with an additional "twist" as compared to \eqref{mptt-sym}, such as a symbol  acting on $F_1(\xi_1,\cdot)$ and  $F_2(\cdot,\eta_2)$,  
or $F_1(\cdot,\xi_2)$ and $F_2(\eta_1,\cdot)$.
A fiber-wise action of this kind was  first studied by Kova\v{c} \cite{vk:tp} and the first author \cite{bernicot:fw} in the one-parameter setting  \eqref{cmsym} and dimension $n=2$. In this paper we address such a  situation  in the multi-parameter setting.

We follow customary practice to model symbols by the convolution-type $P-Q$  operators, see for instance  \cite{vk:tp} or \cite{mtt:unifpara}.
For $k\in\Z$ and $1\leq i\leq 2$, let $\varphi_{i,k}$ and $\psi_{i,k}$ be   smooth  functions adapted in the interval  $[-2^{k+1},2^{k+1}]$, and
let ${\psi_{i,k}}$ vanish on $[-2^{k-100}, 2^{k-100}]$. 
A     function $\rho$  adapted to an interval  $I\subseteq \R$   means a function supported in $I$ and satisfying 
$$\|\partial^\alpha \rho\|_\infty \leq |I|^{-\alpha}$$ for all multi-indices $\alpha$  up to order $N$ for some large $N$; see    \cite{stein}.
Let  $P_{i,k}$ and  $Q_{i,k}$     denote  the one-dimensional Fourier multipliers with  symbols $\varphi_{i,k}$ and $\psi_{i,k}$ respectively, i.e.  
\begin{align*}
P_{i,k} f = f * \widecheck{\varphi}_{i,k}, \quad 
Q_{i,k} f = f * \widecheck{\psi}_{i,k}
\end{align*}
for $f\in \L^{1}_{\textup{loc}}(\R)$.
 When we apply such operators to one-dimensional fibers of a two-dimensional function we use a superscript to denote the fiber on which the action takes place. For instance, 
$$ P_{i,k}^{(1)} F(x_1,x_2) = \big( F(\cdot, x_2) * \widecheck{\varphi}_{i,k}\big ) (x_1), \quad P_{i,k}^{(2)} F(x_1,x_2) = \big( F(x_1,\cdot) * \widecheck{\varphi}_{i,k} \big  ) (x_2),$$
and similarly for  $Q_{i,k}^{(1)} F$ and  $Q_{i,k}^{(2)} F$.
The central objects of this paper are the two-dimensional  bi-parameter bilinear operators
\begin{align*}
& T_1(F_1, F_2)(x) =\sum_{(k,l)\in \Z^2: k\leq l} (Q^{(1)}_{1,k} P_{2,l}^{(2)}F_1)(x)   \, (Q_{2,l}^{(1)}P_{1,k}^{(2)} F_2)(x)\qquad\textup{and}\\
& T_2(F_1,F_2)(x) =\sum_{(k,l)\in \Z^2:k\leq l} (P^{(1)}_{1,k} P_{2,l}^{(2)}F_1)(x)   \, (Q_{2,l}^{(1)}Q_{1,k}^{(2)} F_2)(x),
\end{align*}
defined for test functions $F_1,F_2:\R^2\rightarrow \mathbb{C}$.  
We prove the following bounds in Section \ref{sec:thm1}.
\begin{theorem}
\label{thm:doubletwist} 
The operators $T_1$ and $T_2$
are bounded from $\L^{p_1}(\R^2) \times \L^{p_2}(\R^2)$ to $\L^{p'_3}(\R^2)$ whenever     $1<p_1,p_2<\infty$, $1<p_3'<2$, and $\frac{1}{p_1} + \frac{1}{p_2} = \frac{1}{p_3'}$.
\end{theorem}

Passing to the Fourier side, the operators $T_1$ and $T_2$   can be viewed as     multiplier operators which map the tuple $(F_1,F_2)$ to the two-dimensional function   defined by 
\begin{align}
\label{mult-t}
x \mapsto 
\int_{\R^{4}} \widehat{F_1}(\xi) \widehat{F_2}(\eta) m(\xi,\eta) e^{2\pi i x\cdot (\xi+\eta)} d\xi d\eta
\end{align}
for a suitable bounded function $m$ on $\R^4$. In the case of $T_1$ and $T_2$,  the function  $m$  
satisfies 
\begin{align}\label{symest-bi1}
|\partial^{\alpha}_{(\xi_1,\eta_2)}\partial^{\beta}_{(\xi_2,\eta_1)} m(\xi,\eta)|\leq C  |(\xi_1,\eta_2)|^{-|\alpha|}|(\xi_2,\eta_1)|^{-|\beta|}
\end{align}
for all $\alpha,\beta\in \mathbb{N}_0^2$ up to order $N$ and all $(\xi,\eta)\neq 0$. 
In general, the estimates \eqref{symest-bi1} alone are not sufficient for boundedness of the multiplier operator \eqref{mult-t}. We elaborate on this in Section \ref{subsec:counter} by reducing a special case of \eqref{mult-t} to the aforementioned counterexample 
from  \cite{GK}. 

Here and in the sequel, the notion of bi- and  multi-parameter is related to the number of parameters of frequency scales. For instance, each factor on the right-hand side of \eqref{symest-bi1} gives rise to one parameter.
Given the constraint $k\leq l$, the symbols of $T_1,T_2$ do not split into the tensor products of two symbols  $m_1(\xi_1,\eta_2)$ and $m_2(\xi_2,\eta_1).$ However, bounds for $T_1,T_2$ imply bounds on the multiplier operator \eqref{mult-t} in the case when $m$ is indeed of tensor type, i.e. 
\begin{align}
\label{m-tensor}
m(\xi,\eta) = m_1(\xi_1,\eta_2)m_2(\xi_2,\eta_1),
\end{align}
where  $m_1,m_2$ satisfy the estimates
\begin{align}\label{symbolest}
|\partial_{(\zeta_1,\zeta_2)}^\alpha m_i(\zeta_1,\zeta_2)|\leq C  |(\zeta_1,\zeta_2)|^{-\alpha} 
\end{align}
for all $\alpha\in \mathbb{N}_0^2$ up to order $N$ and all $0\neq (\zeta_1,\zeta_2)\in \R^2$. This can be seen by  the classical cone decomposition.
\begin{corollary}
\label{cor1}
Let $m$ be given as in \eqref{m-tensor}. 
Then the associated  operator  \eqref{mult-t} 
is bounded from $\L^{p_1}  \times \L^{p_2} $ to $\L^{p'_3}$ whenever     $1<p_1,p_2<\infty$, $1<p_3'<2$, and $\frac{1}{p_1} + \frac{1}{p_2} = \frac{1}{p_3'}$. 
\end{corollary}
 The multiplier  operator \eqref{mult-t} with the symbol   \eqref{m-tensor} was suggested   by Camil Muscalu. The case when $m_1$ and $m_2$ are localized to   cones in the frequency plane   is sometimes called     twisted paraproduct.
 A special case  when one of the symbols $m_i$ is constantly equal to one, i.e.
 \begin{align}\label{op:tp}
x \mapsto 
\int_{\R^{4}} \widehat{F_1}(\xi) \widehat{F_2}(\eta) m(\xi_1,\eta_2) e^{2\pi i x\cdot (\xi+\eta)} d\xi d\eta,
 \end{align}
has been previously studied by Kova\v{c} \cite{vk:tp} and the first author \cite{bernicot:fw}. It is a degenerate case of the two-dimensional bilinear Hilbert transform investigated by Demeter and Thiele \cite{dt:2dbht}.   The one-parameter operator  \eqref{op:tp} is known to map  $\L^{p_1} \times \L^{p_2} \rightarrow \L^{p_3'}$ in a larger  range   than stated in  Corollary \ref{cor1}, namely whenever the exponents satisfy
\begin{align*}
   1<p_1,p_2<\infty,\; \frac{1}{2}<p_3'<2,\quad \textup{and} \quad \frac{1}{p_1} + \frac{1}{p_2} = \frac{1}{p_3'}.
\end{align*}
Indeed, the techniques  developed in   \cite{vk:tp} yield bounds whenever $2<p_1,p_2<\infty,\, 1<p_3'<2$. Fiber-wise Calder{\'o}n-Zygmund decomposition from \cite{bernicot:fw}  can be then used to extend the range of exponents.  

Let us also mention that there are further  connections of \eqref{op:tp}  with several objects in ergodic theory. 
Indeed, the techniques developed in \cite{vk:tp}   have been subsequently refined and used to study larger classes of multilinear forms generalizing \eqref{op:tp}, which are motivated by   problems on quantifying norm convergence of ergodic averages: the papers by Kova\v{c} \cite{vk:nvea} and Kova\v{c}, \v{S}kreb, Thiele and the second author \cite{DKST:nvea} study ergodic averages with respect to  two commuting transformations, while  \v{S}kreb    \cite{skreb:nvea} studies certain  cubic averages. Other  applications of   operators related to the twisted paraproduct are in stochastic integration obtained by Kova\v{c} and \v{S}kreb \cite{KS:stochastic} and  also in Euclidean Ramsey theory when investigating patterns in large subsets of the Euclidean space; see for instance  the work by Kova\v{c} and the second author \cite{DK:boxes}.
We also    refer to the paper by Stip\v{c}i{\'c} \cite{Stip}  and the references therein.

More generally,  one can consider the class of  all bilinear operators \eqref{mult-t}    
where   $m$ is  of  the  tensor type
   \begin{align}
   \label{m-tensor2}
   m(\zeta_1,\zeta_2,\zeta_3,\zeta_4)= m_1((\zeta_{a})_{a\in S_1})m_2((\zeta_{a})_{a\in S_2})m_3((\zeta_{a})_{a\in S_3})m_4((\zeta_{a})_{a\in S_4}),
   \end{align}
where   $S_{i}\subseteq\{1,2,3, 4\}$  and  
  ${m}_i$ are symbols on $\R^{|S_i|}$,  each of them satisfying the     estimates analogous to \eqref{symbolest}.  
  Then \eqref{m-tensor} is a special case with  $S_3=S_4=\emptyset$. 
  For the purpose of this paper, let us restrict ourselves to the case when the sets  $S_i$ are pairwise disjoint and consider \eqref{mult-t} with such a  symbol.  
If $S_1=\{1\}$, then due to boundedness of the one-dimensional Calder{\'o}n-Zygmund operators we may replace the function $\widehat{F}_1$  in \eqref{mult-t} by  
\begin{align}
\label{reduce}
m_1(\xi_1)\widehat{F}_1(\xi).
\end{align}
Performing the analogous step for any singleton $S_i$,  we may   reduce   \eqref{mult-t}   to the case where   each non-empty set $S_i$ satisfies $|S_i|\geq 2$. 
 Up to symmetries, mapping properties of the      cases that arise are discussed in  Table \ref{table} below. 
\begin{table}[htb]
\renewcommand{\arraystretch}{1.7}
\centering
\caption{Structure of $m$ and the   range of boundedness} 
\label{table}
\begin{tabular}{|c|c|l|}
\hline
& Structure of $m(\xi,\eta)$ & Known range of boundedness \\
\hline
1& $ m_1(\xi_1,\eta_1)$   & $1<p_1,p_2\leq \infty,\ \frac{1}{2}<p_3'<\infty$ \\ 
\hline
2& $ m_1(\xi_1,\xi_2)$   & $1<p_1,p_2< \infty, \ \frac{1}{2}<p_3'<\infty$\\ 
\hline
3 & $m_1(\xi_1,\eta_2)$   & $1<p_1,p_2< \infty,\ \frac{1}{2}<p_3'<2$ \\
\hline
4& $m_1(\xi_1,\xi_2)  m_2(\eta_1,\eta_2)$   & $1<p_1,p_2< \infty,\ \frac{1}{2}<p_3'<\infty$\\
\hline
5& $m_1(\xi_1,\eta_1) m_2(\xi_2,\eta_2)$  &  $1<p_1,p_2\leq \infty,\ \frac{1}{2}<p_3'<\infty$\\
\hline
6& $m_1(\xi_1,\eta_2)  m_2(\xi_1,\eta_1)$   & $1<p_1,p_2< \infty,\ 1<p_3'<2$ \\
\hline
7& $m_1(\xi_1,\xi_2,\eta_1)$  & $1<p_1,p_2< \infty,\ \frac{1}{2}<p_3'<\infty$ \\
\hline
8& $m_1(\xi_1,\xi_2,\eta_1,\eta_2)$  & $1<p_1,p_2\leq  \infty,\ \frac{1}{2}<p_3'<\infty$ \\
\hline
\end{tabular}
\end{table}

The operator corresponding to  Case 1 is a classical Coifman-Meyer multiplier \cite{cm1,cm2} acting on the first fibers of  the functions $F_1,F_2$. Case 2 is a pointwise product of a linear Fourier multiplier with the identity operator and can be treated analogously as  \eqref{reduce}. The estimates for Case 2 stated in Table \ref{table} are obtained  by H\"older's inequality.
Case 3 is the operator \eqref{op:tp}, which has been studied in \cite{vk:tp}. Case 4 is a pointwise product of two linear operators and the analogous reduction as in \eqref{reduce} applies. Case 5 is a bi-parameter version of a Coifman-Meyer multiplier  and its boundedness follows by   \cite{mptt:bi}. As remarked earlier, in \cite{mptt:bi} the authors are able to handle multipliers $m$ which are not necessarily of tensor type. This  is in   contrast with Case 6, which 
 is the content of  Corollary \ref{cor1}.  It remains to consider  Case 8, which is the classical Coifman-Meyer multiplier, and Case 7.
We prove the following bounds for Case 7 in Section~\ref{sec:cor}.

\begin{theorem} \label{thm:3}
If  $m(\xi,\eta) = m_1(\xi_1,\xi_2,\eta_1)$, then the associated operator \eqref{mult-t} is bounded from   $\L^{p_1} \times \L^{p_2}$ to $\L^{p'_3}$ whenever     $1<p_1,p_2<\infty$, $\frac{1}{2}<p_3'<\infty$, and $\frac{1}{p_1} + \frac{1}{p_2} = \frac{1}{p_3'}$.
\end{theorem}

We emphasize that extending the range of exponents for $p_3'$ in Cases 3 and 6 remains open.
 
Dualizing the operators in  \eqref{mult-t} with   symbols of the form \eqref{m-tensor2}, the corresponding trilinear forms
  are particular examples of  singular Brascamp-Lieb integrals  with several singular  kernels, see  also  the survey \cite{dt:sblsurvey}.  Indeed, they can be represented by the trilinear forms
\begin{align}
\label{form:sbl}
\int_{\R^6} F_1(x+A_1s+B_1t)F_2(x+A_2s+B_2t) F_3(x)K(s,t) dxdsdt
\end{align}
for suitable matrices $A_1,B_1,A_2,B_2\in M_2(\R)$ and a kernel $K$, whose Fourier transform $\widehat{K}$ is of the form  \eqref{m-tensor2}.
In particular, the  operator with \eqref{m-tensor} is associated with a tensor product of two Calder{\'o}n-Zygmund  kernels
\begin{align*}
K(s,t) = K_1(s)K_2(t),
\end{align*}
where $K_1= \widecheck{m}_1 $, $K_2=\widecheck{m}_2$. 
In this case, if one of the kernels $K_1,K_2$ in \eqref{form:sbl} specializes to the Dirac delta, then the object reduces to a one-parameter family of the two-dimensional bilinear Hilbert transforms \cite{dt:2dbht}. 
Studying  \eqref{form:sbl} for an arbitrary choice of matrices $A_i,B_i$ and any Calder{\'o}n-Zygmund kernels $K_1,K_2$ remains an open problem.

Furthermore, the operators $T_1$ and $T_2$ can be viewed as particular  fiber-wise versions of the  paraproducts studied by Muscalu, Tao, and Thiele in \cite{mtt:unifpara}.
Following the setup from  \cite{mtt:unifpara}, let $n\geq 1$ and let $ \Omega_{n} \in \Z^{n}$ be a convex polytope of the form
\begin{align*}
\Omega_n = \{(k_1,\ldots,k_{n}) \in \Z^{n} : k_{i_\alpha} \geq k_{i'_\alpha} \textup{ for all } 1\leq \alpha \leq K \},
\end{align*}
where $i_\alpha,i'_\alpha \in \{1,\ldots, n\}$ and  $K\geq 0$ an  integer. 
Consider the operators which maps an $n$-tuple  $(F_i)_{i=1}^n$  of test functions on $\R^2$ to the two-dimensional function 
\begin{align}
\label{mainform}
x\mapsto  \sum_{(k_1,\ldots,k_{2n})\in \Omega_{2n} }   \prod_{i=1}^n  (Q^{(1)}_{i,k_i}Q^{(2)}_{i+n,k_{i+n}}F_i)(x),
\end{align}  where each $Q_{i,k}$ is a Fourier multiplier with symbol
  $\psi_{i,k}$, which is a bump function adapted in   $[-2^{k+1},2^{k+1}]$ and 
vanishes on  $[-2^{k-100}, 2^{k-100}]$. 
When
$\Omega_{2n}$ is of the form $\Omega_{n} \times \mathbb{Z}^n$, bounds for \eqref{mainform} follow from \cite{mtt:unifpara}.
When $n=1$,   \eqref{mainform} reduces to a classical linear Calder{\'o}n-Zygmund operator. 
When $n=2$, one can classify the cases similarly as in Table \ref{table}. Indeed, this can be seen  by summing over   $\Omega_4$ in at least two parameters $k_i$ in \eqref{mainform}.  Using the fact that $\sum_{k\in \Z}Q_{i,k}$ are linear Calder{\'o}n-Zygmund operators and hence satisfy the desired bounds,   the problem is reduced to objects with the summation over $\Omega_2$, such as $T_1,\,T_2$, and other bi-parameter analogues in \cite{mptt:bi}. 
This yields $\L^{p_1} \times \L^{p_2}$ to $\L^{p'_3}$    bounds whenever     $1<p_1,p_2<\infty$ and $1<p_3'<2$, but the known range may, in particular cases,  be larger.
However, as it is the case for $T_1$ and $T_2$,  the symbols of   \eqref{mainform}  are in general not of tensor type because of the constraint on  $\Omega_n$.

 The main  idea used in the proof of Theorem  \ref{thm:doubletwist}   is to 
 reduce the problem to the vector-valued estimates for the operator  \eqref{op:tp} with one constant symbol. The key step is in observing that   the operator is localized in frequency due to the frequency supports of the fibers of the input  functions. This can be seen in sharp contrast with \eqref{op:tp} itself, where such localization does not occur.  Localization of  the operator allows replacements of some low-frequency projections, acting on the input functions, with the identity operators. This, in turn, allows for  applications of H\"older's inequality.
 We prove Theorem \ref{thm:doubletwist} and Corollary \ref{cor1} in Section \ref{sec:thm1}. 
 In the case of \eqref{op:tp}, quasi-Banach estimates can be proven using the fiber-wise Calder{\'o}n-Zygmund decomposition from \cite{bernicot:fw}. This decomposition does not seem applicable in the context of Theorem \ref{thm:doubletwist}, as the symbols act on both fibers of the input functions. Extending the range of exponents in Theorem \ref{thm:doubletwist} remains an open problem.

Theorem \ref{thm:3} is proven in Section \ref{sec:thm2} and in the Banach case it relies on the bounds for the  operator \eqref{op:tp}. In Theorem \ref{thm2}, the operator acts on only one fiber of the function $F_2$; in this case we are able to use the fiber-wise Calder{\'o}n-Zygmund decomposition to prove quasi-Banach estimates as stated in Theorem \ref{thm:3}.

\subsection*{Multilinear and multi-parameter generalizations}
At present there is only partial understanding of the multilinear generalizations of Theorem \ref{thm:doubletwist}.
In the case of \eqref{m-tensor},   multilinear operators with only one non-constant symbol can be described in the language of  bipartite graphs and   were  studied by Kova\v{c} in \cite{vk:bell} in a    dyadic model. 

In this paper,  we discuss a particular  tri-parameter  trilinear example, which can be seen as a natural generalization of \eqref{m-tensor}.  Let $m_1,m_2,m_3$ be symbols satisfying  \eqref{symbolest} for $1\leq i \leq 3$. Define the trilinear  operator  which maps a triple  $(F_1,F_2,F_3)$ of functions on   $\R^2$ to the two-dimensional function   given by 
\begin{align}  
x \mapsto \int_{\R^6} &\widehat{F}_1(\xi)\widehat{F}_2(\eta)\widehat{F}_2(\tau) m_1(\xi_1,\eta_2)m_2(\eta_1,\tau_2)m_3(\tau_1,\xi_2)  e^{2\pi ix\cdot(\xi+\eta+\tau)}    d\xi d\eta d\tau, \label{tripletwist} 
\end{align}
where $\xi=(\xi_1,\xi_2)$, $\eta=(\eta_1,\eta_2)$, $\tau=(\tau_1,\tau_2)$.  
We prove the following bounds.  
\begin{theorem}
\label{thm2}
The operator \eqref{tripletwist} is bounded from  $\L^{p_1}  \times \L^{p_2}  \times \L^{p_3}  $ to  $\L^{p'_{4}} $ when $1<p_1,p_2,p_3<\infty$, $2<p_4<\infty$,   $\sum_{i=1}^4 \frac{1}{p_1}=1$, and    
$
\frac{1}{p_1}+\frac{1}{p_2}>\frac{1}{2},\, \frac{1}{p_2}+\frac{1}{p_3}>\frac{1}{2},\,  \frac{1}{p_1}+\frac{1}{p_3}>\frac{1}{2}.
$
\end{theorem}
Note that the range in Theorem \ref{thm2}  is non-empty. 
 For example, it contains exponents  in vicinity of $3<p_1=p_2=p_3<4$. 
The proof of Theorem \ref{thm2}  is detailed in Section \ref{sec:thm2} and can be seen as an iteration of the steps used in the proof of Theorem \ref{thm:doubletwist} and its corollary, by gradually reducing to the vector-valued estimates for the operators   with one or more constant symbols. Iterating this procedure and applying estimates for one- and two-parameter operators, which hold in   restricted ranges of exponents, is the reason for further restriction of  the range in 
  Theorem~\ref{thm2}.

   The proof of Theorem \ref{thm2} does not immediately generalize to all higher degrees of multilinearities. Beside facing  the issues with the  exponent range, objects with constant symbols  which arise in the proof may not be   localized in frequency, which prevents further iterations of the approach. 
Similar issues  also arise when trying to generalize   Theorems \ref{thm:doubletwist} and \ref{thm2} to higher dimensions. 
Obtaining  bounds for a larger class of multi-parameter objects, such as multilinear operators in \eqref{mainform}, is closely related to studying a large class of maximally truncated singular integrals, one- and multi-parameter. One such instance is  the maximally truncated one-parameter twisted paraproduct  which maps a tuple   $(F_1,F_2)$ of functions  on $\R^2$ to the two-dimensional function   
 \begin{align}\label{maxtwist}
x \mapsto   \sup_{N>0}\  \Big| \sum_{|k|<N} (P_{1,k}^{(1)} F_1)(x) \, (Q^{(2)}_{1,k} F_2)(x)
\Big|. \end{align}
Its boundedness seems  to be out of reach of the currently available techniques. Further motivation for proving such  maximal estimates  and  also stronger variational estimates  is provided by questions on pointwise convergence of  certain  ergodic averages. 
Establishing  $\L^p$  estimates for the  operator which maps a tuple   $(F_1,F_2)$ of functions  on $\R^2$ to the one-dimensional function  
 \begin{align*} 
x_2 \mapsto    \sup_{N>0} \ \Big  \| \sum_{|k|<N} (P_{1,k}^{(1)} F_1)(x_1,x_2 ) \, (Q^{(2)}_{1,k} F_2) (x_1,x_2 )  \Big \|_{\L^p_{x_1}(\R)}
 \end{align*}
could be considered as  an intermediate case between     \eqref{op:tp} and  \eqref{maxtwist}. 
 Such bounds  also remain  open. \\

\noindent{\bf Notation.}
For two non-negative quantities $A,B$ we  write $A\lesssim B$ if there exists an absolute constant $C$ such that $A\leq CB$. If $C$ depends on the parameters  $P_1,\ldots, P_n$, we write $A\lesssim_{P_1,\ldots,P_n} B$.\\

\noindent {\bf Acknowledgments.} 
The authors thank
Cristina Benea, Vjekoslav Kova\v{c}, Camil Muscalu and Christoph Thiele for several motivating discussions.
The first author is  supported by ERC project FAnFArE no.637510. 
The second author acknowledges  the hospitality of Universit{\'e} de Nantes while this research was performed.

\section{The bi-parameter bilinear operators $T_1$ and $T_2$} 
\label{sec:thm1}
This section is devoted to the proof of Theorem \ref{thm:doubletwist}.   Throughout this section   we will   use the shorthand notation  $k\ll  l$ to denote $k< l-200$. Similarly,  $k \gg  l$ will denote $k>  l+200$ and $k\sim  l$ will mean  $l- 200\leq k \leq l+200$. 

 It will be evident  from the proof   that the argument  will not rely on the particular choice of  the bump functions $\varphi_{i,k},\psi_{i,k}$ as long as they satisfy the   assumptions   in the definition of $T_1$ and $T_2$. For simplicity of notation, we shall therefore only discuss the case $\varphi_{1,k}=\varphi_{2,k}=\varphi_k$ and $\psi_{1,k}=\psi_{2,k}=\psi_k$ for each $k\in \Z$. We  shall  also write  $P_{1,k}=P_{2,k}=P_k$ and $Q_{1,k}=Q_{2,k}= Q_k$.

\subsection{Boundedness of $T_1$} 
First we split the summation over $k\leq l$  into the regimes  where $k\sim l$ and  $k\ll l$ respectively, i.e. 
\begin{align}\label{T1split}
T_1(F,G) = \sum_{\substack{(k,l)\in \Z^2\\ l-200\leq k\leq l}}({Q}^{(1)}_{k} {P}^{(2)}_{l}  F_1)\,  ({Q}^{(1)}_{l}  {P}^{(2)}_{k} F_2)  + \sum_{\substack{(k,l)\in \Z^2\\ k\ll l}}({Q}^{(1)}_{k} {P}^{(2)}_{l}  F_1)\, ({Q}^{(1)}_{l}  {P}^{(2)}_{k} F_2).
\end{align}

Bounds for the  sum over $l-200\leq k\leq l$ follow  by  Cauchy-Schwarz. Indeed, using Cauchy-Schwarz in $l\in \Z$ we pointwise bound  this term   as   
\begin{align*}
\Big |  \sum_{-200\leq  s \leq 0} \sum_{l\in \Z}  ({Q}^{(1)}_{l+s} {P}^{(2)}_{l}  F_1)\,  ({Q}^{(1)}_{l}  {P}^{(2)}_{l+s} F_2) \Big |  \leq  \sum_{-200\leq  s \leq 0}   \| {Q}^{(1)}_{l+s} {P}^{(2)}_{l}   F_1 \|_{\ell^2_l(\Z)}  \| {Q}^{(1)}_{l}  {P}^{(2)}_{l+s} F_2   \|_{\ell^2_l(\Z)}.
\end{align*}
We consider the product of terms on the right-hand side for each fixed $-200\leq s\leq 0$.
To estimate the $\L^{p_3'}$ norm of the product we 
apply H\"older's inequality and use bounds for the classical square function.  We obtain
\begin{align*} &\| \| {Q}^{(1)}_{l+s} {P}^{(2)}_{l}   F_1 \|_{\ell^2_l(\Z)}  \| {Q}^{(1)}_{l}  {P}^{(2)}_{l+s} F_2   \|_{\ell^2_l(\Z)}\|_{\L^{p_3'}(\R^2)}\\
& \leq  \| {Q}^{(1)}_{l+s} {P}^{(2)}_{l}   F_1  \|_{\L^{p_1}(\ell^2_l)} \;\|  {Q}^{(1)}_{l}  {P}^{(2)}_{l+s} F_2   \|_{\L^{p_2}(\ell^2_l)}  \lesssim_{p_1,p_2}  \|F_1\|_{\L^{p_1}(\R^2)} \|F_2\|_{\L^{p_2}(\R^2)} \end{align*}
whenever $1/p_3'=1/p_1+1/p_2$ and $1<p_1,p_2<\infty$.
 In the end it remains to sum the individual contributions of these finitely many terms.

It remains to consider the case $k\ll l$ in \eqref{T1split}. By duality it suffices to study  the corresponding trilinear form and   show  
\begin{align*}
\Big |\sum_{(k,l)\in \Z^2: k\ll l}  \int_{\R^2} ({Q}^{(1)}_{k} {P}^{(2)}_{l}  F_1)\,  ({Q}^{(1)}_{l}  {P}^{(2)}_{k} F_2) \,F_3   \Big |  \lesssim_{p_1,p_2,p_3} \|F_1\|_{\L^{p_1}(\R^2)} \|F_2\|_{\L^{p_2}(\R^2)} \|F_3\|_{\L^{p_3}(\R^2)}
\end{align*}
for any choice of exponents $1<p_1,p_2<\infty$, $2<p_3<\infty$ with  $1/p_1+1/p_2+1/p_3=1$.
By the frequency supports of  $F_1$,  $F_2$, the   form on the left-hand side  can be written up to a constant as 
\begin{align}\label{form-11}
& \sum_{(k,l)\in \Z^2: k\ll l}   \int_{\R^2} ({Q}^{(1)}_{k} {P}^{(2)}_{l}  F_1)\,  ({Q}^{(1)}_{l}  {P}^{(2)}_{k} F_2) \,(\mathcal{Q}^{(1)}_{l} \mathcal{P}^{(2)}_{l} F_3),  
\end{align}
where $\mathcal{P}_l$ and $\mathcal{Q}_l$ are Fourier multipliers with symbols adapted to $[- 2^{l+3}, 2^{l+3}]$, the symbol of $\mathcal{P}_l$ is constant on $[- 2^{l+2}, 2^{l+2}]$,  and the symbol of $\mathcal{Q}_l$ vanishes on $[- 2^{l-90}, 2^{l-90}]$.

Then  we  write $P_l = \varphi_l(0)I + (P_l-\varphi_l(0)I)$ where $I$ is the identity operator, 
and plug this decomposition into the form \eqref{form-11}. This yields  
\begin{align}\label{m-e-terms}
 \eqref{form-11} =  \sum_{(k,l)\in \Z^2:k\ll l}   \mathcal{M}_{k,l}   \quad +\quad    \sum_{(k,l)\in \Z^2:k\ll l}   \mathcal{E}_{k,l} ,
\end{align}
where we have defined
\begin{align*}
  \mathcal{M}_{k,l} &= \int_{\R^2} ({Q}^{(1)}_{k}    F_1)\,  ({Q}^{(1)}_{l}  {P}^{(2)}_{k} F_2) \, (\varphi_l(0) \mathcal{Q}^{(1)}_{l} \mathcal{P}^{(2)}_{l} F_3),\\
    \mathcal{E}_{k,l} &=  \int_{\R^2}  ({Q}^{(1)}_{k} (P^{(2)}_l-\varphi_l(0)I^{(2)}) F_1)\, ( {Q}^{(1)}_{l}  {P}^{(2)}_{k} F_2) \,(\mathcal{Q}^{(1)}_{l} \mathcal{P}^{(2)}_{l} F_3).
\end{align*}

First we consider the term involving $\mathcal{M}_{k,l}$.  Since $|\varphi_l(0)|\leq 1$, up  to redefining $\mathcal{P}_l$ we may assume that $\varphi_l(0)=1$ for each $l$. 
We split the summation as 
\begin{align} \label{split-sum}
\sum_{(k,l)\in \Z^2:k\ll l}  \mathcal{M}_{k,l}  \quad = \quad   \sum_{(k,l)\in \Z^2}   \mathcal{M}_{k,l}  \quad -\quad  \sum_{(k,l)\in \Z^2:k\sim  l}   \mathcal{M}_{k,l}\quad  - \quad  \sum_{(k,l)\in \Z^2:k\gg l}    \mathcal{M}_{k,l} . 
\end{align}
 By the frequency support of the first fibers of the functions $F_i$ it follows that the term over $k\gg l$ vanishes. 
To estimate  the term with summation over   $k\sim l$ we use  H\"older's inequality in $l$ and in  the integration, yielding 
\begin{align*}
\Big |  \sum_{l\in \Z}\sum_{s\sim 0}    \mathcal{M}_{l+s,s} \Big |\leq \sum_{s\sim 0}  \| Q_{l+s}^{(1)} F_1 \|_{\L^{p_1}(\ell^\infty_l)} \|{Q}^{(1)}_{l}  {P}^{(2)}_{l+s} F_2 \|_{\L^{p_2}(\ell^2_l)} \|   \mathcal{Q}^{(1)}_{l}  \mathcal{P}^{(2)}_{l} F_3 \|_{\L^{p_3}(\ell^2_l)}.
\end{align*} 
It  remains to   use bound on the one-dimensional  maximal function and the two square functions, which hold uniformly in $s\sim 0$, and finally sum in $s.$

Thus,  it suffices to estimate  the case when the sum is unconstrained, i.e. over $(k,l)\in \Z^2$. By Cauchy-Schwarz in $l$ and  H\"older's inequality in the integration we estimate  
\begin{align}\label{form:unconstrained}
&\Big |  \sum_{(k,l)\in \Z^2} \mathcal{M}_{k,l}\Big | \leq \Big \|   \sum_{k\in \Z}  ({Q}^{(1)}_{k}    F_1)\,  ({Q}^{(1)}_{l}  {P}^{(2)}_{k} F_2)  \Big \|_{\L^{p_3'}(\ell^2_l)}\, \|  \mathcal{Q}^{(1)}_{l} \mathcal{P}^{(2)}_{l} F_3  \|_{\L^{p_3}(\ell^2_l)},
\end{align} 
where  $1/p_3+1/p_3'=1$. The second term on the right-hand side is a classical square function. To bound  the first term we use 
$\L^{p_1}(\R^2) \times \L^{p_2}(\ell^2)\rightarrow\L^{p_3'}(\ell^2)$ vector-valued estimates for the operator \eqref{op:tp}, which hold whenever $1<p_3'<2$ and $1<p_1,p_2<\infty$. These vector-valued  estimates are obtained by freezing the function $F_1\in \L^{p_1}$ and using the linear inequalities of Marcinkiewicz and Zygmund \cite{mz}, together with scalar-valued boundedness of \eqref{op:tp}. We obtain 
\begin{align*}
 \Big \|   \sum_{k\in \Z}  ({Q}^{(1)}_{k}    F_1)\,  ({Q}^{(1)}_{l}  {P}^{(2)}_{k} F_2) \Big \|_{\L^{p_3'}(\ell^2_l)} \lesssim_{p_1} \|F_1\|_{\L^{p_1}(\R^2)} \|{Q}^{(1)}_{l} F_2\|_{\L^{p_2}(\ell^2_l)} \lesssim_{p_1,p_2}  \|F_1\|_{\L^{p_1}(\R^2)} \|F_2\|_{\L^{p_2}(\R^2)},
\end{align*}
where  we have used the Littlewood-Paley inequality for the last bound.
This yields the desired estimate for the first term in \eqref{m-e-terms}.

It remains to estimate the second term in   \eqref{m-e-terms}. 
By frequency consideration in the second fiber, one has  
\begin{align}\label{qtildetilde}
 \sum_{(k,l)\in \Z^2: k\ll l}  \mathcal{E}_{k,l} =  c \sum_{(k,l)\in \Z^2: k\ll l}  \int_{\R^2} ({Q}^{(1)}_{k} \widetilde{\mathcal{Q}}^{(2)}_l F_1)\, ( {Q}^{(1)}_{l}  {P}^{(2)}_{k} F_2) \,(\mathcal{Q}^{(1)}_{l} \mathcal{P}^{(2)}_{l} F_3),
\end{align}
where $c$ is an absolute constant and  $\widetilde{\mathcal{Q}}_l$ is a   Fourier multiplier  with symbol  adapted to $[-2^{l+3},2^{l+3}]$ which vanishes at the origin. Indeed,   note that this is the case both when $\varphi_l(0)=0$ and when $\varphi_l(0)\neq 0$.
We split the summation in the regions where $(k,l)\in \Z^2$, $k\ll l$ and $k\gg l$. By 
   the analogous  considerations as  in the paragraphs from \eqref{split-sum} to \eqref{form:unconstrained} we note that   it suffices to  instead   consider the case when the sum is unconstrained, i.e.
\begin{align*}
 \sum_{(k,l)\in \Z^2}   \int_{\R^2} ({Q}^{(1)}_{k} \widetilde{\mathcal{Q}}^{(2)}_l F_1)\, ( {Q}^{(1)}_{l}  {P}^{(2)}_{k} F_2) \,(\mathcal{Q}^{(1)}_{l} \mathcal{P}^{(2)}_{l} F_3).
\end{align*}
By the  Cauchy-Schwarz inequality in $l$ and H\"older's inequality in the integration we bound the last display by 
\begin{align*}
 \Big \|   \sum_{k\in \Z} ({Q}^{(1)}_{k}  \widetilde{\mathcal{Q}}_l^2  F_1 )\, ( {Q}^{(1)}_{l}  {P}^{(2)}_{k} F_2 ) \Big \|_{\L^{p_3'}(\ell^2_l)}\,  \|  \mathcal{Q}^{(1)}_{l} \mathcal{P}^{(2)}_{l} F_3   \|_{\L^{p_3}(\ell^2_l)}.
\end{align*} 
The second term is a square function. 
Bounds for the first term follow from $\L^p(\ell^2) \times \L^q(\ell^2) \rightarrow \L^r(\ell^2)$ vector-valued estimates for the twisted paraproduct \eqref{op:tp} and two applications of the Littlewood-Paley inequality. The vector valued estimates which we need in this case are a consequence of scalar boundedness of the operator \eqref{op:tp} and a result  by Grafakos and Martell  \cite[Theorem 9.1]{gm}.  
This yields the desired bound for  the second term in \eqref{m-e-terms} and in turn establishes the claim for  $T_1$.

 \subsection{Boundedness of $T_2$}
The proof for $T_2$ proceeds in the analogous way as  the proof  for $T_1$ and we only sketch the necessary ingredients. 
 By duality it suffices to bound
 \begin{align*}
\Big | \sum_{(k,l)\in \Z^2: k\leq l}  \int_{\R^2}  ({P}^{(1)}_{k} {P}^{(2)}_{l}  F_1)\,  ({Q}^{(1)}_{l}  {Q}^{(2)}_{k} F_2) \,F_3   \Big |  \lesssim_{p_1,p_2,p_3} \|F_1\|_{\L^{p_1}(\R^2)} \|F_2\|_{\L^{p_2}(\R^2)} \|F_3\|_{\L^{p_3}(\R^2)}
\end{align*}
for any choice of exponents $1<p_1,p_2<\infty$, $2<p_3<\infty$ with  $1/p_1+1/p_2+1/p_3=1$.
The case when  $l-200\leq k\leq l$ is  bounded by H\"older's inequality, so it suffices to consider  $k\ll l$.

By  frequency considerations, the form on the left-hand side is a constant multiple of 
\begin{align*}
  \sum_{(k,l)\in \Z^2: k\ll l}  \int_{\R^2} ({P}^{(1)}_{k} {P}^{(2)}_{l}  F_1)\,  ({Q}^{(1)}_{l}  {Q}^{(2)}_{k} F_2) \, (\mathcal{Q}^{(1)}_l \mathcal{P}^{(2)}_l F_3)\,  
\end{align*}
for frequency  projections $\mathcal{Q}_l$ and $\mathcal{P}_l$ as in  \eqref{form-11}. 
As above we  split $P_l = \varphi(0)I+ (P_l-\varphi(0)I)$. By  considerations as in  the discussion after \eqref{m-e-terms}
  it remains  to estimate   the analogue of the term associated with $\mathcal{M}_{k,l}$, i.e.
\begin{align*}
 \sum_{(k,l)\in \Z^2}   \int_{\R^2} ({P}^{(1)}_{k}   F_1)\,  ({Q}^{(1)}_{l}  {Q}^{(2)}_{k} F_2) \, (\varphi(0)\mathcal{Q}^{(1)}_l \mathcal{P}^{(2)}_l F_3),
\end{align*} 
and the analogue of   the term associated with $\mathcal{E}_{k,l}$, i.e.
\begin{align*}
\sum_{(k,l)\in \Z^2}   \int_{\R^2} ({P}^{(1)}_{k}  \widetilde{\mathcal{Q}}^{(2)}_{l}    F_1)\,  ({Q}^{(1)}_{l}  {Q}^{(2)}_{k} F_2) \, (\mathcal{Q}^{(1)}_l \mathcal{P}^{(2)}_l F_3),
\end{align*}
where $\widetilde{\mathcal{Q}}_{l} $ is as in \eqref{qtildetilde}.  
The proofs for each of these terms  proceed analogously as for \eqref{form:unconstrained} and  \eqref{qtildetilde} respectively, reducing  to   vector-valued estimates for the operator \eqref{op:tp}.

\begin{remark}
An alternative way to prove bounds for $T_2$ is to deduce them from the bounds for $T_1$ via a  telescoping identity, which "swaps" the $P$- and $Q$-type operators. Indeed, this can be achieved in a special case when   $T_1$ and  $T_2$ are related by a condition on the bump functions as in Proposition \ref{mainprop-tri} in Section \ref{sec:thm2} below. Then, one can deduce a general case of $T_2$ from the  special case by an averaging argument. We perform such arguments in a trilinear tri-parameter  setting in  Section \ref{sec:thm2} below.
\end{remark}

\subsection{Proof of Corollary \ref{cor1}}
\label{sec:conedec}
This corollary can be deduced from Theorem \ref{thm:doubletwist} by  a classical cone decomposition, as performed  for instance in \cite{Muscalu}. For completeness we outline the relevant steps.

Let $m$ be a function on $\R^2\setminus \{0\}$ satisfying
\begin{align*}
|\partial^\alpha m(\zeta)| \lesssim |\zeta|^{-|\alpha|}
\end{align*}
for all $\alpha$ up to a large finite order and all $\zeta\neq 0$ in $\R^2$.
Let $\varphi$ be a smooth function supported in $[-2,2]$ and constantly equal to $1$ on $[-1,1]$. Let $\psi=\varphi-\varphi(2 \cdot )$. Then $\sum_{k\in \Z}\psi(2^{-k}\tau)=0$ for each $\tau\neq 0$. We can write 
\begin{align*}
m(\zeta) = \sum_{(k,l)\in \Z^2}m(\zeta)\psi(2^{-k}\zeta_1) \psi(2^{-l}\zeta_2).
\end{align*}
Splitting the sum into regions when     $k\leq  l$ and  $k> l$, respectively, and  summing in the smaller parameter we obtain 
\begin{align*}
m(\zeta) = \sum_{k\in \Z}m(\zeta)\varphi(2^{-k}\zeta_1) \psi(2^{-k}\zeta_2) + \sum_{k\in \Z}m(\zeta)\psi(2^{-k}\zeta_1) \varphi(2^{-(k-1)}\zeta_2).
\end{align*}
We consider the first sum; the second is treated analogously.

Note that for each $k\in \Z$, the summand is supported in 
$$\{\zeta\in \R^2: |\zeta_1|\leq 2^{k+1},\,2^{-k-1}\leq |\zeta_2| \leq 2^{k+1}\}.$$  
The smooth restriction of $m$ to that region can be decomposed into a double Fourier series, which yields
\begin{align*}
\sum_{k\in \Z}m(\zeta)\varphi(2^{-k}\zeta_1) \psi(2^{-k}\zeta_2) = \sum_{(n_1,n_2)\in\Z^2} \sum_{k\in \Z} C_{n_1,n_2}^k\varphi(2^{-k}\zeta_1)e^{c\pi i2^{-k}\zeta_1 n_1} \psi(2^{-k}\zeta_2)e^{c\pi i2^{-k}\zeta_2 n_2}
\end{align*}
where $c$ is a fixed constant and  the Fourier coefficients $C_{n_1,n_2}^k$ satisfy
$$|C_{n_1,n_2}^k|\lesssim (1+|n_1|)^{-N} (1+|n_2|)^{-N}$$
for any  $N>0$ up to a large order, uniformly in $k\in \Z$. For details we refer to   \cite{Muscalu}, Chapter 2.13.

Normalizing the bump functions and the coefficients,   the last display can be written as an absolute  constant times 
 \begin{align*}
\sum_{(n_1,n_2)\in \Z^2}(1+|n_1|)^{-2}(1+|n_2|)^{-2}  \sum_{k\in \Z} \widetilde{C}_{n_1,n_2}^k \varphi_{k,n_1}(\zeta_1)\psi_{k,n_2}(\zeta_2),
\end{align*}
where \[\varphi_{k,n}(\zeta_1) = c_1 \varphi(\zeta_1)e^{c\pi i \zeta_1 n_1},\;\psi_{k,n_2}(\zeta_2)=c_2 \psi(\zeta_2)e^{c\pi i \zeta_2 n_2},\]
and the constants $c_1,c_2$ are such that both functions are adapted to $[-2^{k+1},2^{k+1}]$. Moreover,  $\varphi_{k,n_1}(0)$ is the same constant for each $k\in \Z$, $\psi_{k,n_2}$ vanishes in $[-2^{k-1},2^{k-1}]$, and $|\widetilde{C}_{n_1,n_2}^k|\leq 1.$ 
This reduces the matters to considering symbols for each fixed $n_1,n_2$ with bounds uniform in $n_1,n_2$. Note that the coefficients $\widetilde{C}_{n_1,n_2}^k$ can assumed to be equal to $1$ by subsuming them  into the definitions of $\psi_{k,n_2}$.

We perform this decomposition for the symbols $m_1$ and $m_2$ in \eqref{m-tensor}. Then the bounds for the associated  operator follow from bounds on $T_1$ and $T_2$. 

\subsection{A counterexample} \label{subsec:counter}
In this section we show that the estimates \eqref{symest-bi1} alone are not sufficient for boundedness of the operator \eqref{mult-t}. We restrict ourselves to the Banach regime. More precisely, we show that there exists a bounded function $m$ on $\R^4$ satisfying 
\begin{align*}
|\partial^{\alpha}_{(\xi_1,\eta_2)}\partial^{\beta}_{(\xi_2,\eta_1)} m(\xi,\eta)|\lesssim_{\alpha,\beta}  |(\xi_1,\eta_2)|^{-|\alpha|}|(\xi_2,\eta_1)|^{-|\beta|}
\end{align*}
for all multi-indices $\alpha,\beta \in \mathbb{N}_0$ such that the associated operator \eqref{mult-t} does not satisfy   $\L^{p_1} \times \L^{p_2}$ to  $\L^{p_3'}$ estimates for any exponents  $1< p_1,p_2,p_3< \infty$. 

 To see this we first recall a  result from \cite{GK}; see  the remark at the end of Section~3 in \cite{GK}.  Existence is shown of a symbol $m_0$ on $\R^2\setminus\{0\}$ satisfying
 \begin{align*}
|\partial^{\alpha}_{\xi_1}\partial^{\beta}_{\eta_1} m_0(\xi_1,\eta_1)|\lesssim_{\alpha,\beta}|\xi_1|^{-\alpha}|\eta_1|^{-\beta} 
\end{align*}
for all multi-indices $\alpha,\beta\in \mathbb{N}_0$, 
 such that  the one-dimensional    operator mapping a tuple  $(f_1,f_2)$ of functions on $\R$ to a one-dimensional function
\begin{align}\label{gk:operator}
x \mapsto \int_{\R^2} \widehat{f_1}(\xi_1)\widehat{f_2}(\eta_1) m_0(\xi_1,\eta_1) e^{2\pi ix(\xi_1+\eta_1)} d\xi_1 d\eta_1,
\end{align}
does not satisfy any $\L^{p_1}\times \L^{p_2}$ to $\L^{p_3'}$  bounds.

To show that the estimates \eqref{symest-bi1} are in general not  sufficient  we reduce a special case of \eqref{mult-t} to this counterexample. 
Let us define
\begin{align*}
m(\xi_1,\xi_2,\eta_1,\eta_2) = m_0(\xi_1,\eta_1) \widetilde{m}(\xi_1,\eta_2) \widetilde{m}(\eta_1,\xi_2),
\end{align*}
where $\widetilde{m}$ is a smooth symbol on $\R^2\setminus \{0\}$  supported in the cone $\{(\zeta_1,\zeta_2) : |\zeta_2|\lesssim |\zeta_1| \}$, satisfying $\widetilde{m}(\zeta_1,0)=1$ whenever $|\zeta_1|\neq 0$,  and  
$$|\partial^\alpha \widetilde{m}(\zeta_1,\zeta_2)| \lesssim_\alpha |(\zeta_1,\zeta_2)|^{-|\alpha|}$$ 
for all $\alpha\in \mathbb{N}_0^2$ and all $(\zeta_1,\zeta_2)\neq 0$.
 Then $m$ satisfies  \eqref{symest-bi1} and we have $m(\xi_1,0,\eta_1,0)=m_0(\xi_1,\eta_1)$.
  We dualize the operator \eqref{mult-t} associated  with this multiplier and consider the corresponding trilinear form, which reads 
 \begin{align}\label{dualize-counter}
 \int_{\R^{4}} \widehat{F_1}(\xi) \widehat{F_2}(\eta) \widehat{F_3}(\xi+\eta) m(\xi,\eta)  d\xi d\eta.
 \end{align}
Let $\lambda>0$ and let $1<p_1,p_2,p_3<\infty$ be such that $1/p_1+1/p_2+1/p_3=1$. For $1\leq i \leq 3$ we  set  
$$F_i(x_1,x_2) = f_i(x_1) \lambda^{-1/p_i}\varphi(\lambda^{-1}x_2), $$ where $f_i$ and $\varphi$ are one-dimensional smooth compactly supported functions and $\widehat{\varphi}\geq 0$.
Plugging these particular  functions $F_i$ into the trilinear form \eqref{dualize-counter} we obtain 
\begin{align}\nonumber
\lambda^{-1} \int_{\R^4}  &\widehat{f}_1(\xi_1)\lambda \widehat{\varphi}(\lambda \xi_2)\widehat{f_2}(\eta_1)\lambda \widehat{\varphi}(\lambda \eta_2) \widehat{f_3}(\xi_1+\eta_1)\lambda \widehat{\varphi}(\lambda (\xi_2+\eta_2)) \\
 & m(\xi_1,\xi_2,\eta_1,\eta_2)   d\xi _1d\xi_2d\eta_1d\eta_2.  \label{counterex}
\end{align}
By rescaling in $\xi_2$ and $\eta_2$ it suffices to consider
\begin{align*}
\int_{\R^4}\widehat{f}_1(\xi_1)\widehat{\varphi}(\xi_2)\widehat{f_2}(\eta_1) \widehat{\varphi}( \eta_2) \widehat{f_3}(\xi_1+\eta_1) \widehat{\varphi}(\xi_2+\eta_2)m(\xi_1,\lambda^{-1}\xi_2,\eta_1,\lambda^{-1}\eta_2)    d\xi _1d\xi_2d\eta_1 d\eta_2  .  
\end{align*}
 Letting $\lambda\rightarrow \infty$ and then integrating in $\xi_2,\eta_2$ we obtain  a non-zero constant multiple of 
\begin{align*}
\int_{\R^2}\widehat{f}_1(\xi_1)\widehat{f_2}(\eta_1) \widehat{f_3}(\xi_1+\eta_1)  m_0(\xi_1,\eta_1)   d\xi_1d\eta_1,
\end{align*}
which  can be recognized as  \eqref{gk:operator} paired with a function $f_3$. Therefore, the form in the last display   does not satisfy any $\L^p$ estimates. Since   $\|F_i\|_{\L^{p_i}(\R^2)}$ equals   $\|f_i\|_{\L^{p_i}(\R)}\|\varphi\|_{\L^{p_i}(\R)}$, it follows that \eqref{counterex} does not satisfy any $\L^p$ estimates as well.

\section{A  one-parameter bilinear operator}
\label{sec:cor}

This section is devoted to the proof of Theorem \ref{thm:3}. By a three-dimensional  analogue of the cone decomposition  outlined  in Section \ref{sec:conedec} it  suffices to estimate the operators mapping $(F_1,F_2)$   to  two-dimensional functions given by
\begin{align} \label{opnew1}
x\mapsto \sum_{k\in\Z} (Q_{k}^{(1)}P^{(2)}_{1,k} F_1)(x)\, (P_{2,k}^{(1)}F_2)(x),\\ \label{opnew2}
x\mapsto \sum_{k\in\Z} (P_{1,k}^{(1)}P^{(2)}_{2,k} F_1)(x)\, (Q_{k}^{(1)}F_2)(x),
\end{align}
and
\begin{align}
 \label{opnew3}
x\mapsto \sum_{k\in\Z} (P_{1,k}^{(1)}Q^{(2)}_{k} F_1)(x)\, (P_{2,k}^{(1)}F_2) (x),
\end{align} 
where $P_{i,k}$, $Q_{k}$ are Fourier multipliers with symbols  $\varphi_{i,k}$ and $\psi_k$, respectively, which are  adapted to $[-2^{k+1},2^{k+1}]$, and in addition $\psi_k$ vanishes on $[-2^{k-1},2^{k-1}]$.

We will first prove bounds for each of them in the open Banach range. Then we will  extend the range with a fiber-wise  Calder\'on-Zygmund decomposition from \cite{bernicot:fw}.

\subsection{Boundedness in the open Banach range}
 In this section we  show that  the operators \eqref{opnew1}, \eqref{opnew2} and \eqref{opnew3} 
are bounded from $\L^{p_1}\times \L^{p_2} $ to $\L^{p'_3}$ whenever $1<p_1,p_2,p_3<~\infty$ and $1/p_1+1/p_2 +1/p_3=1$.

To bound the first operator \eqref{opnew1},  we first reduce to the case when the symbol of  $P_{2,k}$ is supported in $[-2^{k-99},2^{k-99}]$. Indeed, if this is not the case    we split 
$$P_{2,k} = \widetilde{P}_{2,k} + Q_{2,k},$$ 
where the symbols of $\widetilde{P}_{2,k}$ and $Q_{2,k}$ are supported in $[-2^{k-99},2^{k-99}]$ and $[-2^{k+1}, -2^{k-100}]\cup [2^{k-100},2^{k+1}]$, respectively. We split the operator accordingly and note that the term with $Q_{2,k}$  immediately reduces to two square functions by Cauchy-Schwarz.  

Now we dualize and consider the corresponding trilinear form. By the frequency localization in the first fibers of the functions, it suffices to show
\begin{align*}
\Big | \sum_{k\in\Z} \int_{\R^2} (Q_{k}^{(1)}P^{(2)}_{1,k} F_1)\, (\widetilde{P}_{2,k}^{(1)}F_2)\,  (\mathcal{Q}_k^{(1)}F_3) \Big | \lesssim_{p_1,p_2,p_3} \|F_1\| _{\L^{p_1}(\R^2)} \|F_2\|_{\L^{p_2}(\R^2)} \|F_3\|_{\L^{p_3}(\R^2)}
\end{align*} 
for any choice of exponents $1<p_1,p_2,p_3<\infty$ and $1/p_1+1/p_2 +1/p_3=1$.
Here $\mathcal{Q}_k$ is adapted to $[-2^{k+3},2^{k+3}]$ and the corresponding symbol vanishes at the origin.
This estimate  follows immediately by H\"older's inequality in $k$ and in the integration, together with  bounds  on the square and maximal function. 

For the second operator \eqref{opnew2} we proceed in the same way; this time reducing to the case when the symbol of  $P_{1,k}$ is supported in $[-2^{k-99},2^{k-99}]$.

It remains to consider \eqref{opnew3}.   By considerations as above, we may assume that $P_{1,k}$ is supported in $[-2^{k-99},2^{k-99}]$. 
By duality it suffices to study the 
  corresponding trilinear  form
\begin{align}\label{trilinearform}
\sum_{k\in \Z } \int_{\R^2}(P_{1,k}^{(1)}Q^{(2)}_{k} F_1)\, (P_{2,k}^{(1)}F_2) \, F_3 =   c \sum_{k\in \Z}\int_{\R^2} (P_{1,k}^{(1)}{Q}_{k}^{(2)}F_1)\,   \mathcal{Q}^{(2)}_{k}\Big( (P_{2,k}^{(1)} F_2)\,    F_3 \Big),
\end{align}
where $c$ is a constant and $\mathcal{Q}_k$ satisfies the properties as in the previous display.
By    Cauchy-Schwarz   in $k$ and H\"older's inequality in the integration, we bound the last display by
\begin{align*}
 \|  P_{1,k}^{(1)}Q_{k}^{(2)}F_1   \|_{\L^{p_1}(\ell^2_k)}   \Big \| \mathcal{Q}^{(2)}_{k}\Big( (P_{2,k}^{(1)} F_2)  \,  F_3 \Big ) \Big \|_{\L^{p'_1}(\ell^2_k)}. 
\end{align*}
The first term is a square function bounded in the full   range. 
For the second term, we  note that by Fatou's lemma, it suffices to restrict the $\ell^2_k$ norm to  $|k|\leq K$ for some $K>0$ and prove   bounds uniform in $K$. 
By Kintchine's inequality, this term    reduces to   having to estimate
\begin{align*}
\Big \|\sum_{|k|\leq K} a_k  \mathcal{Q}^{(2)}_{k}\Big( (P_{2,k}^{(1)} F_2)    F_3\Big )  \Big \|_{\L^{p_1'}(\R^2)} 
\end{align*} 
for uniformly bounded coefficients $a_k$. 
Dualizing with   $G\in \L^{p_1}(\R^2)$ it suffices to show
$$ \Big |  \sum_{|k|\leq K} \int_{\R^2} a_k  (\mathcal{Q}^{(2)}_{k}G)\,  (P_{2,k}^{(1)} F_2)  \,  F_3 \,  \Big | \lesssim_{p_1,p_2,p_3} \|G\| _{\L^{p_1}(\R^2)}  \|F_2\|_{\L^{p_2}(\R^2)} \|F_3\|_{\L^{p_3}(\R^2)} .$$
This estimate holds in the range  $1<p_1,p_2<\infty$,   $2<p_3<\infty$, and H\"older scaling, since   the left-hand side is an operator of the form \eqref{op:tp} paired with the function $F_3$.

 To obtain bounds for \eqref{opnew3} in the full range, it suffices to show estimates  with  $1<p_1,p_3<\infty$ and $2<p_2<\infty$.  The left-hand side of  \eqref{trilinearform} can be written as 
 \begin{align*}
 \sum_{k\in\Z} \int_{\R^2} (P_{1,k}^{(1)}Q^{(2)}_{k} F_1)\, (P_{2,k}^{(1)}F_2) \, \mathcal{P}^{(1)}_k F_3
 \end{align*}
 for a   multiplier $\mathcal{P}_k$ with symbol adapted to $[-2^{k+3},2^{k+3}]$. 
Writing $P_{2,k} = \varphi_{2,k}(0)I + (P_{2,k}-\varphi_{2,k}(0)I)$   we reduce the last display to
\begin{align*}
&  \sum_{k\in \Z} \int_{\R^2}  \,  (P_{1,k}^{(1)} Q_{k}^{(2)} F_1)\, F_2  \, (\varphi_{2,k}(0)\mathcal{P}^{(1)}_{k}   F_3)\, + \,    \sum_{k\in \Z} \int_{\R^2}   (P_{1,k}^{(1)} Q_{k}^{(2)} F_1)\, ( \mathcal{Q}_{k}^{(1)}F_2)\,  ( \mathcal{P}^{(1)}_{k}   F_3)  
\end{align*}
for a suitably defined  $\mathcal{Q}_{k}$ whose symbol vanishes at the origin.
The first term is analogous to the left-hand side of \eqref{trilinearform}, satisfying the estimates with  $1<p_1,p_3<\infty$ and $2<p_2<\infty$. The second term has two   $Q$-type operators   and bounds in the full range follow by H\"older's inequality and bounds on the maximal and square functions.

\subsection{Fiber-wise Calder{\'o}n-Zygmund decomposition}
In this section we show that the  operators \eqref{opnew1},  \eqref{opnew2} and \eqref{opnew3} 
are bounded from $\L^{p_1}\times \L^{p_2} $ to $\L^{p'_3}$ whenever $1<p_1,p_2<\infty$, $1/2<p_3'<\infty$ and $1/p_1+1/p_2=1/p_3'$.

 By real bilinear interpolation, it suffices to prove weak-type bounds with $p_3'< 1$ for the  operators \eqref{opnew1}, \eqref{opnew2} and \eqref{opnew3}. Each of them is of the form
$$ T(F_1,F_2)(x)=\sum_{k\in\Z} (P_{1,k}^{(1)}P^{(2)}_{3,k} F_1)(x)\, (P_{2,k}^{(1)}F_2)(x),$$
where $P_{i,k}$ are Fourier multipliers with symbols adapted to $[-2^{k+1},2^{k+1}]$. 
We  are going to apply the fiberwise Calder\'on-Zygmund from  \cite{bernicot:fw} to the function $F_2$.

It suffices to show that for any $1<p_1,p_2<\infty
$ and $1/2<p_3'< 1$, the operator $T$ satisfies the weak $\L^{p_1}\times \L^{p_2}  \rightarrow \L^{p_3',\infty}$ estimates. Fix $F_1\in \L^{p_1}(\R^2)$ and $F_2\in \L^{p_2}(\R^2)$.  By homogeneity we may assume 
$$\|F_1\|_{\L^{p_1}(\R^2)}=\|F_2\|_{\L^{p_2}(\R^2)} = 1.$$
Our goal is  to show that for every $\lambda>0$
$$ |\{x\in \R^2: \ |T(F_1,F_2)(x)|>\lambda\}| \lesssim_{p_1,p_2} \lambda^{-p_3'}.$$
We write $x=(x_1,x_2)$.
Fix $\lambda>0$ and perform a fiber-wise Calder{\'o}n-Zygmund decomposition of $F_2(\cdot, x_2)$ for fixed $x_2\in \R$ at level $\lambda^{p_3'/p_2}$. This yields  functions $g_{x_2}$ and atoms $a_{i,x_2}$ supported on disjoint dyadic intervals intervals $I_{i,x_2}$ such that
$F_2(\cdot, x_2) = g_{x_2} + \textstyle{\sum_{i}} a_{i,x_2}$ and  for all $x_2$ 
\begin{align*}
\|g_{x_2}\|_{\L^{p_2}(\R)}\leq \|F_2(\cdot, x_2)\|_{\L^{p_2}(\R)}  \quad \textup{and} \quad \|g_{x_2}\|_{\L^\infty(\R)}\lesssim \lambda^{p_3'/p_2}.
\end{align*}
Moreover,  for all $i$, the atom $a_{i,x_2}$ has vanishing mean on $I_{i,x_2}$,      
$$\|a_{i,x_2}\|_{\L^{p_2}(I_{i,x_2})}\lesssim \lambda^{p_3'/p_2}|I_{i,x_2}|^{1/p_2},$$
and the   intervals satisfy
\begin{equation} \sum_i |I_{i,x_2}| \lesssim \lambda^{-p_3'} \|F_2(\cdot,x_2)\|_{\L^{p_2}(\R)}^{p_2}. \label{eq:covering} \end{equation}

First we consider the good part $T(F_1,g)$, where $g=(x_1,x_2)\mapsto g_{x_2}(x_1)$. We have 
\[\|g\|_{\L^{p_2}(\R^2)}\leq \|F_2\|_{\L^{p_2}(\R^2)} =1\quad \textup{and} \quad \|g\|_{\L^\infty(\R^2)} \lesssim \lambda^{p_3'/p_2} .\]
 Therefore, for $q_2>p_2$ we get 
$$ \|g\|_{\L^{q_2}(\R^2)} \lesssim_{p_2} \lambda^{p_3'(1/p_2-1/q_2)}.$$
By the  boundedness in the Banach range obtained in  the previous section we obtain 
\begin{align*}
|\{x\in \R^2 :  |T(F_1,g)(x)|>\lambda\}| & \leq \lambda^{-s} \|T(F_1,g)\|^{s}_{\L^s(\R^2)} \\
& \lesssim_{p_1,p_2}\lambda^{-s} \|F_1\|_{\L^{p_1}(\R^2)}^s \|g\|_{\L^{q_2}(\R^2)}^s \lesssim \lambda^{-s + sp_3'(1/p_2-1/q_2)} = \lambda^{-p_3'},
\end{align*}
where $1<q_2,s<\infty$ are any exponents that satisfy  $1/{p_1} + 1/{q_2} = 1/{s}$, $q_2>p_2$ and $(p_1,q_2,s)$ belong  to the open Banach range. This is the desired estimate on the level set.

It remains to consider the bad part $T(F_1,b)$, where $b=(x_1,x_2)\mapsto \sum_i a_{i,x_2}(x_1)$.  Denote
$$E=\bigcup_{x_2\in \R} \bigcup_{i}\, (   2I_{i,x_2} \times \{x_2\}),$$
where $2I_{i,x_2}$ is  the interval with the same center as $I_{i,x_2}$ but twice the length.
Note  that
\begin{align*}
|E|\lesssim \int_{\R} |\cup_i2 I_{i,x_2}| dx_2 \lesssim \lambda^{-p_3'} \int_{\R} \|F_2(\cdot, x_2)\|_{\L^{p_2}(\R)}^{p_2} dx_2   
\lesssim \lambda^{-p_3'} \|F_2\|_{\L^{p_2}(\R^2)}^{p_2} = \lambda^{-p_3'},
\end{align*}
 where we have used \eqref{eq:covering}.
Therefore, it remains to estimate
$
|\{x \not \in E :\, |T(F_1,b)(x)| > \lambda   \}|.
$
By the definition we have
 \begin{align*}
  P^{(1)}_{2,k} b(x) = \sum_i P_{2,k} [a_{i,x_2}](x_1).
 \end{align*}
Fix $x_2\in \R$ and take $x_1\in \R \setminus \cup_i2 I_{i,x_2}$.  Denote by $c_{i,x_2}$ the  center of the interval $I_{i,x_2}$. 
Since $P_{2,k}$ is a convolution operator with some smooth function $\rho_k$ (with its Fourier transform adapted to $[-2^{k+1}, 2^{k+1}]$ up to a   large order), using the mean zero property of $a_{i,x_2}$, we obtain 
  \begin{align*}
  |P_{2,k} [a_{i,x_2}](x_1)| &\leq \int_{I_{i,x_2}} |a_{i,x_2} (w)| |\rho_k(x_1-w) - \rho_k(x_1-c_{i,x_2})| dw\\
  &\lesssim \int_{I_{i,x_2}} |a_{i,x_2} (w)|\, 2^{2k}|I_{i,x_2}|   \big (1+ 2^{k}|x_1-c_{i,x_2}|  \big )^{-100} dw\\
  &\lesssim \lambda^{p_3'/p_2}  \, 2^{2k}|I_{i,x_2}|^2 \big (1+ 2^{k}|x_1-c_{i,x_2}|  \big )^{-100}.
  \end{align*}
  Here we used   
  $$\|a_{i,x_2}\|_{\L^1(\R)} \leq |I_{i,x_2}|^{1/p_2'} \|a_{i,x_2}\|_{\L^{p_2}(\R)}  \lesssim \lambda^{p_3'/p_2} |I_{i,x_2}|.$$
  Therefore, since $|x-c_{i,x_2}|\geq |I_{i,x_2}|/2$ we have
$$ \sum_{k\in \Z} 2^{2k} |I_{i,x_2}|^2  \Big (1+ \frac{|x_2-c_{i,y}|}{2^{-k}} \Big )^{-100} \lesssim  \Big (1+ \frac{|x_2-c_{i,y}|}{|I_{i,x_2}|} \Big )^{-2}. $$
Therefore, we have showed that for  $x_1\in \R \setminus \cup_i2 I_{i,x_2}$ it holds 
 \begin{align*}
 |T(F_1,b)(x)|  \lesssim \lambda^{p_3'/p_2} {\mathcal M}(F_1)(x) H(x), 
 \end{align*}
where ${\mathcal M}$ is the Hardy-Littlewood maximal function and 
  \begin{align*}
  H(x) = \sum_i \Big (1+ \frac{|x_1-c_{i,y}|}{|I_{i,y}|} \Big )^{-2} .
  \end{align*}
Next we use boundedness of the Marcinkiewicz functions associated with a disjoint collection of the intervals $(I_{i,x_2})_i$, i.e.
 \begin{align*}
\Big \| \sum_{i} \Big (1+ \frac{|x_1-c_{i,x_2}|}{|I_{i,x_2}|} \Big )^{-2}  \Big \|_{\L_{x_1}^{p_2}(\R)} \lesssim_{p_2}   \Big( \sum_i |I_{i,x_2}| \Big)^{1/p_2}.
\end{align*}
(See   \cite{stein0} or Grafakos, Exercise 4.6.6.) This yields 
\begin{align*}
\|H\|_{\L^{p_2}(\R^2)}     \lesssim_{p_2} \Big \|  \Big( \sum_{i} |I_{i,x_2}| \Big)^{1/p_2}\Big \|_{\L_{x_2}^{p_2}(\R)}   \lesssim  \lambda^{-p_3'/p_2} \| \| F_2(x_1,x_2)\|_{\L^{p_2}_{x_1}(\R)} \|_{\L^{p_2}_{x_2}(\R)} \lesssim \lambda^{-p_3'/p_2}.
\end{align*}
Therefore,  
\begin{align*}
|\{x \not \in E :\, |T(F_1,b)(x)| > \lambda   \}| & \lesssim |\{x\in{\mathbb R}^2 : \   \lambda^{p_3'/p_2} {\mathcal M}(F_1)(x) H(x) \gtrsim \lambda \}| \\
& = |\{x\in{\mathbb R}^2: \  {\mathcal M}(F_1)(x) H(x) \gtrsim \lambda^{1-p_3'/p_2}\}|  \\
& \lesssim \lambda^{-p_3'+(p_3')^2/p_2} \|\mathcal{M}(F_1)\|_{\L^{p_1}(\R^2)}^{p_3'} \| H \|_{\L^{p_2}(\R^2)}^{p_3'} \\
& \lesssim_{p_1,p_2} \lambda^{-p_3'+(p_3')^2/p_2-(p_3')^2/p_2}  \lesssim \lambda^{-p_3}.
\end{align*}

Summarizing, we have shown that  the operator $T$ satisfies the weak  inequality $\L^{p_1} \times \L^{p_2} \rightarrow \L^{p_3',\infty}$. By interpolation it is bounded in the range claimed by Theorem \ref{thm2}.
 
\begin{remark} We emphasize that we are able to make use of the fiberwise Calder\'on-Zygmund in this case (and not in the case of Theorem \ref{thm:doubletwist}) because  the operators  \eqref{opnew1}, \eqref{opnew2} and \eqref{opnew3} act on only   fiber of the function $F_2$. 
Alternatively, one could apply the classical two-dimensional Calder{\'o}n-Zygmund decomposition by decomposing the function $F_1$.
\end{remark}

\begin{remark}
An alternative approach to prove quasi-Banach estimates for \eqref{opnew1} and \eqref{opnew2}, which avoids    the use of fiberwise Calder\'on-Zygmund decomposition and in fact  proves a stronger claim, namely,  that \eqref{opnew1} and \eqref{opnew2} are   bounded with values in the Hardy space $H^{p_3'} \subseteq L^{p_3'}$, is the following. Note that the operators  \eqref{opnew1} and  \eqref{opnew2} are localized in frequency in the first coordinate at   scale   $k$. So we can use the argument from \cite{Benea-Muscalu} about the inequality 
$ \|f\|_{p} \lesssim \|S(f)\|_{p}$
for $0< p \leq 1$, where  $S$ is the Littlewood-Paley square function. 
 More precisely, following the scalar case of \cite[Theorem 3.1]{Benea-Muscalu}, 
  one obtains for $0< p_3' \leq 1$   the bound
$$ \| \sum_{k\in\Z} (Q_{k}^{(1)}P^{(2)}_{1,k} F_1)\, (P_{2,k}^{(1)}F_2) \|_{\L^{p_3'}(\R^2)} \lesssim \| (Q_{k}^{(1)}P^{(2)}_{1,k} F_1)\, (P_{2,k}^{(1)}F_2) \|_{\L^{p_3'}(\ell^2_k)}.$$
To apply the  arguments from  \cite{Benea-Muscalu} it is  important is to obtain the localized estimates  which are reduced through duality by factorizing a $Q_k$ operator on the dual function,  which is possible in this situation. Then the square function on the right-hand side is easily estimated by the product of a maximal and square function, both of them which are bounded. A similar reasoning can be done for the second operator of type \eqref{opnew2}. In that way we recover the $\L^{p_3'}$--boundedness of operators \eqref{opnew1} and \eqref{opnew2} and we also prove boundedness in the Hardy space $H^{p_3'}$. 
 \end{remark}

\section{A tri-parameter trilinear operator}
\label{sec:thm2}

The goal of this section it to    prove Theorem \ref{thm2}.  Throughout this section   we will   use the shorthand notation  $k\ll  l$ to denote $k< l-50$. Similarly,  $k \gg  l$ will denote $k>  l+50$ and $k\sim  l$ will mean  $l- 50\leq k \leq l+50$. 

For  $k\in \Z$ and $1\leq i \leq 3$ let  $Q_{i,k}$  be Fourier multipliers with symbols $\psi_{i,k}$. For $k\in \Z$ and   $4\leq j \leq 6$  let $P_{j,k}$  be Fourier multipliers with symbols $\varphi_{j,k}$. Here $\psi_{i,k}$ and $\varphi_{j,k}$   are smooth one-dimensional functions. In what follows we will consider various   classes of symbols and we will apply more assumptions on them as we proceed.
We define the trilinear   operators $U_1$ and $U_2$ acting on two-dimensional functions $F_1,F_2,F_3:\R^2\rightarrow\mathbb{C}$ by
\begin{align*}
U_1(F_1,F_2,F_3)(x)& =  \sum_{(k,l,m)\in \Z^3}   (Q^{(1)}_{1,k} P^{(2)}_{4,m}  F_1)(x)\,  (Q^{(1)}_{2,l}  P^{(2)}_{5,k} F_2)(x)\, (Q^{(1)}_{3,m}  P^{(2)}_{6,l} F_3)(x),  \\
U_2(F_1,F_2,F_3)(x)& =   \sum_{(k,l,m)\in \Z^3}  (P^{(1)}_{5,k} P^{(2)}_{4,m}  F_1)(x)\,  (Q^{(1)}_{2,l}  Q^{(2)}_{1,k} F_2)(x)\, (Q^{(1)}_{3,m}  P^{(2)}_{6,l} F_3)(x).
\end{align*}
We keep in mind that they depend on the particular choice of the functions $\psi_{i,k},\varphi_{j,k}$, but we have suppressed that in the notation.

Let the exponents $p_1,p_2,p_3,$ and $p_4'$ satisfy the assumptions stated in   Theorem \ref{thm2}. 
The first step in the proof of Theorem \ref{thm2} is 
to decompose the  symbols $m_1,m_2,m_3$ as in Section \ref{sec:conedec}. 
This gives that   it suffices to prove  $\L^{p_1} \times\L^{p_2} \times\L^{p_3} $ to $\L^{p'_4} $  estimates for $U_1$ and $U_2$ under the assumptions that for $1\leq i \leq 3$,
  $\psi_{i,k}$   is adapted to $[-2^{k+1},2^{k+1}]$, and for $4\leq j \leq 6$,
$\varphi_{j,k}$  is  adapted to $[-2^{k+1},2^{k+1}]$. Moreover, each $\psi_{i,k}$ vanishes on $[-2^{k-1},2^{k-1}]$ and $\varphi_{j,k}(0)=\varphi_{j,0}(0)$  for each $k$.
Indeed, note that while cone decomposition gives $8$ terms,  all other terms that are of the form of $U_1$ and $U_2$ with  $P$- and $Q$-type operators interchanged  can be deduced from the bounds on $U_1$ and $U_2$ by symmetry considerations, that is, up to interchanging the role of the functions $F_i$ and considering their transposes $(x,y)\mapsto F_i(y,x)$.

Bounds for $U_1$ and $U_2$ resulting after a cone decomposition  will be deduced from Lemmas~\ref{mainprop-tri}-\ref{lemma:tri-2}. Lemma \ref{mainprop-tri} concerns a particular case of $U_1$.
\begin{lemma}  
\label{mainprop-tri}
Let $p_1,p_2,p_3$ and $p_4'$ be   exponents as   in Theorem \ref{thm2}.  
For $1\leq i \leq 3$  and $k\in \Z$  let   ${\psi}_{i,k}$ be a   smooth    function supported in $[-2^{k+3},2^{k+3}]$, which vanishes on $[-2^{k-3},2^{k-3}]$. Assume that for each $1\leq i \leq 3$, $\psi_{i,k}$ is of the form
\begin{align}\label{psiphicond}
{\psi}_{i,k} = {\varphi}_{i,k} - {\varphi}_{i,k-1},
\end{align}
 where ${\varphi}_{i,k}$  is a bump function adapted to  $[-2^{k+3},2^{k+3}]$. 
  For $4\leq j \leq 6$ and $k\in \Z$ let  
$\varphi_{j,k}$ be   adapted to $[-2^{k+4},2^{k+4}]$  and assume that the function 
$\psi_{j,k} = \varphi_{j,k} - \varphi_{j,k-1}$
is supported in $[-2^{k+4},2^{k+4}]$ and vanishes on $ [-2^{k-4},2^{k-4}]$.   

Under these assumptions on $\psi_{i,k}$ and $\varphi_{j,k}$, the associated operator $U_1$
 is bounded from $\L^{p_1} \times\L^{p_2} \times\L^{p_3} $ to $\L^{p'_4} $. 
\end{lemma}

Next we state Lemmas \ref{lemma:tri-1} and \ref{lemma:tri-2},  which  concern operators of the form  \eqref{tripletwist} 
with one constant symbol. In particular, these  results will be used in the proof of Lemma \ref{mainprop-tri}.
\begin{lemma}
\label{lemma:tri-1}
For  $k\in \Z$ let $Q_{1,k},Q_{2,k},P_{5,k}$ and $P_{6,k}$   be  Fourier multipliers with symbols adapted to $[-2^{k+10},2^{k+10}] $. Further, assume that  the symbols  of $Q_{1,k},Q_{2,k}$ vanish on 
$[-2^{k-10},2^{k-10}]$. 
Then the trilinear operator which maps   $(F_1,F_2,F_3)$   to the two-dimensional function   given by 
 \begin{align*}
 x\mapsto  \sum_{(k,l)\in \Z^2:k \ll l}   (    Q^{(1)}_{1,k} F_1)(x)\, ( Q^{(1)}_{2,l}P^{(2)}_{5,k}F_2)(x)\, ( P^{(2)}_{6,l} F_3) (x) 
\end{align*} 
 is bounded from $\L^{p_1} \times\L^{p_2} \times\L^{p_3} $ to $\L^{p'_4} $ 
in the range   $1<p_1,p_2,p_3<\infty$, $2<p_4<\infty$,  $\sum_{i=1}^4 \frac{1}{p_i} =1$,  and  $\frac{1}{p_1}+\frac{1}{p_2} >\frac{1}{2}$. 
\end{lemma}

\begin{proof}[Proof of Lemma \ref{lemma:tri-1}]
It will be clear from the proof   that the argument  will not depend on the particular choice of the frequency projections satisfying the requirements from the lemma, so for simplicity of the notation, we only discuss the special cases $P_{5,k}=P_{6,k}=P_k$,  $Q_{1,k}=Q_{2,k}=Q_k$.

By duality it suffices to consider the corresponding quadrilinear form
\begin{align*}
 \sum_{(k,l)\in\Z^2:k \ll l}  \int_{\R^2}  (    Q^{(1)}_{k} F_1)\, ( Q^{(1)}_{l}P^{(2)}_{k}F_2)\, ( P^{(2)}_{l} F_3)  \,   F_4.
\end{align*}
By  the frequency localization of the functions $F_i$ in the first fibers, to prove estimates for the quadrilinear form
it suffices to study
\begin{align} \label{splitm22}
 \sum_{(k,l)\in\Z^2:k \ll l}  \mathcal{M}_{k,l} \quad = \quad   \sum_{(k,l)\in \Z^2}   \mathcal{M}_{k,l}\quad  - \quad  \sum_{(k,l)\in \Z^2: k \sim l}  \mathcal{M}_{k,l} \quad -\quad   \sum_{(k,l)\in \Z^2: k \gg l}   \mathcal{M}_{k,l},
\end{align}
where we have defined
\[ \mathcal{M}_{k,l} = \int_{\R^2} (    Q^{(1)}_{k} F_1)\, ( Q^{(1)}_{l}P^{(2)}_{k}F_2)\,  \mathcal{Q}^{(1)}_{l}\Big(   ( P^{(2)}_{l} F_3) \,  F_4\Big).\]
Here $\mathcal{Q}_l$ is  a Fourier multiplier with a symbol  which is  adapted to $[-2^{l+12},2^{l+12}]$ and vanishes on $[-2^{l-12},2^{l-12}]$.

Note that the sum over $k\gg l$  in \eqref{splitm22} vanishes due to frequency supports.
To bound the sum with   $k\sim l$ we use H\"older's inequality to obtain 
\begin{align*}
&\Big| \sum_{l\in \Z} \sum_{s\sim 0}    \mathcal{M}_{l+s,l}  \Big|\\
& \leq    \sum_{s\sim 0}  \|   Q^{(1)}_{l+s} F_1 \|_{\L^{p_1}(\ell^\infty_l)}  \|   Q^{(1)}_{l}P_{l+s}^{(2)} F_2   \|_{\L^{p_2}(\ell^2_l)}  \Big  \|  \mathcal{Q}^{(1)}_{l}\Big(   ( P_{l}^{(2)} F_3)  \,  F_4\Big) \Big \|_{\L^r(\ell^2_l)}
\end{align*}
whenever $1/p_1+1/p_2+1/r =1$ and $1< p_1,p_2,r< \infty$. For a fixed $s$, bounds  for  the  first two terms   follow by boundedness of the maximal and square functions.
The third  term satisfies  
\begin{align}\label{kin-twisted}
 \Big \|  \mathcal{Q}^{(1)}_{l}\Big(   ( P_{l}^{(2)} F_3)  \,  F_4\Big) \Big \|_{\L^r(\ell^2_l)} \lesssim_{p_3,p_4} \|F_3\|_{\L^{p_3}(\R^2)} \|F_4\|_{\L^{p_4}(\R^2)}
\end{align}
whenever $1<p_3,r<\infty$, $2<p_4 <\infty$  and $1/r = 1/p_3 + 1/p_4$. 
 Indeed, this follows by Kintchine's inequality to linearize the square-sum and then use bounds for \eqref{op:tp}. In the end, it remains to sum  in $s\sim  0$.

Therefore, it suffices to bound  the form in \eqref{splitm22}   with  $(k,l)\in\Z^2$.   
In this case we note that the third factor in $\mathcal{M}_{k,l}$ does not depend on $k$. This allows to   apply Cauchy-Schwarz in $l$ and H\"older's inequality in the integration, yielding
\begin{align*}
& \Big| \sum_{(k,l)\in \Z^2}  \mathcal{M}_{k,l}  \Big| \leq  \Big\|    \sum_{k\in \Z}  (  Q^{(1)}_{k} F_1)\, ( Q_{l}^{(1)}P_{k}^{(2)} F_2)  \Big \|_{\L^s(\ell^2_l)} \Big \|    \mathcal{Q}^{(1)}_{l}\Big(    ( P^{(2)}_{l}F_3)   \,  F_4 \Big )\Big    \|_{\L^r(\ell^2_l)} 
\end{align*}
whenever $1/r+1/s=1$, $1<r,s<\infty$. 
The second factor on the right-hand side equals  \eqref{kin-twisted}. From the known bounds, we obtain the condition 
$1/p_3+1/p_4 = 1/r$ and $1<p_3,r<\infty$, $2<p_4<\infty$.   The first factor maps 
$\L^{p_1}(\R^2)\times \L^{p_2}(\ell^2) \rightarrow \L^{s}(\ell^2)$ whenever
$1/p_1+1/p_2 = 1/s$ and $1<p_1,p_2<\infty$, $1/s>1/2$. 
This follows from vector-valued estimates for  \eqref{op:tp} as discussed in the display following \eqref{form:unconstrained}. 
\end{proof}

\begin{lemma}
\label{lemma:tri-2}
For $k\in \Z$, let  $Q_{1,k},Q_{2,k}$  be Fourier multipliers with symbols  $\sigma_{1,k},\sigma_{2,k}$ which  are functions supported in $[-2^{k+10},2^{k+10}]$ and  vanish  on $[-2^{k-10}, 2^{k-10}]$. 
For $k\in \Z$ let   $P_{5,k},P_{6,k}$  be Fourier multipliers with symbols  $\rho_{5,k},\rho_{6,k}$, which  are
    functions adapted to $[-2^{k+10},2^{k+10}]$.
    
Assume that for each $i\in \{1,2,5,6\}$, these functions are of the form 
\[\sigma_{i,k} = \rho_{i,k}-\rho_{i,k-1},\] where  $\sigma_{4,k},\sigma_{5,k}$   are further functions supported in $[-2^{k+10},2^{k+10}],$ which vanish  on $[-2^{k-10}, 2^{k-10}]$, while   $\rho_{1,k},\rho_{2,k}$    are
    functions adapted to $[-2^{k+10},2^{k+10}]$. 
Then the trilinear operator which maps   $(F_1,F_2,F_3)$   to the two-dimensional function   given by   
 \begin{align}\label{form2-bis}
x\mapsto  \sum_{(k,l)\in \Z^2:k \gg l}   (    Q^{(1)}_{1,k} F_1)(x)\, ( Q^{(1)}_{2,l}P^{(2)}_{5,k}F_2)(x)\, ( P^{(2)}_{6,l} F_3)  (x)
 \end{align}
 is bounded from $\L^{p_1} \times\L^{p_2} \times\L^{p_3} $ to $\L^{p'_4} $
in the range      $1<p_1,p_2,p_3<\infty$, $2<p_4<\infty$, $\sum_{i=1}^4 \frac{1}{p_i} =1$,  and $\frac{1}{p_2}+\frac{1}{p_3} >\frac{1}{2}$.
\end{lemma}

\begin{proof}[Proof of Lemma \ref{lemma:tri-2}]
Let us augment the definitions in the statement of the lemma by setting $Q_{i,k}$ and $P_{i,k}$ to be Fourier multipliers with the respective symbols   $\sigma_{i,k}$ and $\rho_{i,k}$
for any $i\in \{1,2,5,6\}$. 
  Observe that because of the  condition on the bump functions     we have  the telescoping identity 
 \begin{align}
\sum_{k=k_0}^{k_1} P_{i,k} f \, {Q}_{j,k} g + \sum_{k=k_0}^{k_1}   Q_{i,k} f\, {P}_{j,k-1} g  = P_{i,{k_1}} f\, {P}_{j,{k_1}} g - P_{i,{k_0-1}} f\, {P}_{j,{k_0-1}} g \label{telescoping}
\end{align}
for any   $f,g\in \L^{1}_{\textup{loc}}(\R)$ and $i,j\in \{1,2,5,6\}$. Also note 
   that letting $k_0\rightarrow -\infty$ and  $k_1\rightarrow \infty$, the right-hand side becomes the pointwise product $fg$.
   
Our goal is to deduce the claim from  Lemma \ref{lemma:tri-1} by two applications of the   telescoping identity \eqref{telescoping}.  
We write the   sum over $k\gg l$ in \eqref{form2-bis} as 
$\sum_{k\in \Z}\sum_{l\in \Z: l\ll k}  $
and use  \eqref{telescoping} in $l$. This gives  
\begin{align}
 \label{tel:main}
\eqref{form2-bis} = &- \sum_{(k,l)\in \Z^2: k\gg l}  (    Q^{(1)}_{1,k} F_1)\, ( P^{(1)}_{2,l-1}P^{(2)}_{5,k}F_2)\, ( Q^{(2)}_{6,l} F_3)  \\
& +\sum_{k\in \Z}   
(    Q^{(1)}_{1,k} F_1)\, ( P^{(1)}_{2,k-49}P^{(2)}_{5,k}F_2)\, ( P^{(2)}_{6,k-49} F_3)    \label{tel:boundary}
\end{align}
 Indeed, \eqref{tel:boundary} is the boundary term at $l=k-49$, while the term at $-\infty$ vanishes. 
 
Let us   consider  \eqref{tel:boundary}.
We dualize it and write the corresponding form up to a constant as
\begin{align*} \nonumber
 \sum_{k\in \Z}   \int_{\R^2}  \mathcal{Q}^{(1)}_k \Big( (    Q^{(1)}_{1,k} F_1)\, ( P^{(1)}_{2,k-49}P^{(2)}_{5,k}F_2) \Big )\, ( P^{(2)}_{6,k-49} F_3)\, F_4 ,
\end{align*}
where  $ \mathcal{Q}_k$ has a symbol adapted to $[-2^{k+12},2^{k+12}]$ which vanishes on $[-2^{k-12},2^{k-12}]$. Note that by the support assumption on $\sigma_{2,k}$ we necessary have that $\rho_{2,k}(0)$ is the same constant $c_0\in \R$ for each $k\in \Z$. 
Then we write  $P_{2,k-49} = c_0I+(P_{2,k-49}-c_0I)$ and split  the operator accordingly.  By the  frequency support information in the  first fibers  it suffices to bound    the terms
\begin{align}\label{erroryy}
& \sum_{k\in \Z}   \int_{\R^2}  c_0\, \mathcal{Q}^{(1)}_k \Big( (    Q^{(1)}_{1,k} F_1)\, ( P^{(2)}_{5,k}F_2) \Big )\, ( P^{(2)}_{6,k-9} F_3)\, F_4  \quad \textup{and}\\
&  \sum_{k\in \Z}  \int_{\R^2}     \mathcal{Q}^{(1)}_k \Big( (    Q^{(1)}_{1,k} F_1)\, ( \widetilde{\mathcal{Q}}^{(1)}_{k}P^{(2)}_{5,k}F_2) \Big )\, ( P^{(2)}_{6,k-9} F_3)  \, F_4,\label{errory}
\end{align}
where the symbol of    $\widetilde{ \mathcal{Q}}_k$ is adapted in $[-2^{k+100},2^{k+100}]$ and  vanishes at the  origin.  
The desired bounds for the first term  \eqref{erroryy} follow from  
\begin{align}\nonumber
&\Big | \sum_{k\in \Z}   \int_{\R^2}   (    Q^{(1)}_{1,k} F_1)\, ( P^{(2)}_{5,k}F_2) \,\mathcal{Q}^{(1)}_k \Big(  ( P^{(2)}_{6,k-49} F_3) \, F_4  \Big ) \Big |\\ \label{lemmaproof}
& \leq   \|  (    Q^{(1)}_{1,k} F_1)\, ( P^{(2)}_{5,k}F_2) \|_{\L^r(\ell^2_k)}    \Big  \|   \mathcal{Q}^{(1)}_k \Big(  ( P^{(2)}_{6,k-49} F_3) \, F_4  \Big ) \Big \|_{\L^s(\ell^2_k)} 
\end{align}
whenever $1/r+1/s=1$ and  $1<r,s<\infty$. 
For the first term in \eqref{lemmaproof} we use H\"older's inequality to obtain 
\begin{align*}
  \|  (    Q^{(1)}_{1,k} F_1)\, ( P^{(2)}_{5,k}F_2) \|_{\L^r(\ell^2_k)}  \leq  \|   Q^{(1)}_{1,k} F_1  \|_{\L^{p_1}(\ell^2_k)}   \| P^{(2)}_{5,k} F_2\|_{\L^{p_2}(\ell^\infty_k)}  \lesssim_{p_1,p_2} \| F_1\|_{\L^{p_1}(\R^2)}  \| F_2\|_{\L^{p_2}(\R^2)}. 
\end{align*}
whenever 
$1/r=1/p_1+1/p_2$,   $1<p_1,p_2<\infty$.  For the last inequality we have used bounds on the maximal and square functions.
For the second term in \eqref{lemmaproof} we   use  Kintchine's inequality  to linearize the square-sum.  Then we apply bounds for the corresponding operator \eqref{op:tp}, yielding
$$ \Big  \|   \mathcal{Q}^{(1)}_k \Big(  ( P^{(2)}_{6,k-49} F_3) \, F_4  \Big ) \Big \|_{\L^s(\ell^2_k)}  \lesssim_{p_3,p_4} \|F_3\|_{\L^{p_3}(\R^2)} \|F_4\|_{\L^{p_4}(\R^2)}  $$
whenever 
$1/s=1/p_3+1/p_4$,  $1<p_3<\infty$, and $2<p_4<\infty$.
This yields the desired estimate.

 For   \eqref{errory} we proceed in the analogous way, this time estimating  $ \widetilde{\mathcal{Q}}^{(1)}_k P^{(2)}_{5,k} F_2$ by the Hardy-Littlewood maximal function.

It remains to estimate \eqref{tel:main}, for which we still want to switch the projections in $k$. 
Now we write   the double  sum  over $k\gg l$ as 
$\sum_{l\in \Z}\sum_{k\in \Z: k\gg l}  $
and telescope \eqref{tel:main} in $k$, giving  
\begin{align}\label{main3}
 \eqref{tel:main}  = &  \sum_{(k,l)\in\Z^2:k\gg l}   (    P^{(1)}_{1,k-1} F_1)\, ( P^{(1)}_{2,l-1}Q^{(2)}_{5,k}F_2)\, ( Q^{(2)}_{6,l} F_3)  , \\ 
 \label{error3}
 &+ \sum_{l\in \Z}   (    P^{(1)}_{1,l+51} F_1)\, ( P^{(1)}_{2,l-1}P^{(2)}_{5,l+51}F_2)\, ( Q^{(2)}_{6,l} F_3)  ,\\
 \label{twistederror}
 &- \sum_{l\in \Z} (    F_1)\, ( P^{(1)}_{2,l-1}F_2)\, ( Q^{(2)}_{6,l} F_3)  .
\end{align}
The   term \eqref{twistederror} is the boundary term at $k=\infty$, which is just a pointwise product of $F_1$ with an operator of the form \eqref{op:tp}. We obtain $\L^{p_1}\times \L^{p_2}\times \L^{p_3}\rightarrow \L^{p_4'}$ estimates in the range $1<p_1,p_2,p_3,p_4<\infty$, $1/p_2+1/p_3>1/2$, whenever the exponents satisfy the  H\"older scaling.
 The  term \eqref{main3} follows from Lemma \ref{lemma:tri-1} and symmetry considerations, i.e. after interchanging the role of $F_1$ and $F_3$ and the role of the first and second fibers. We obtain estimates in the range $1<p_1,p_2,p_3<\infty$, $2<p_4<\infty$, $1/p_2+1/p_3>1/2$.
 Finally, note  that  
\begin{align}\label{errorx}
\eqref{error3}  = &  \sum_{l\in \Z}   (    P^{(1)}_{1,l+51} F_1)\, ( P^{(1)}_{2,l-1}P^{(2)}_{5,l-49}F_2)\, ( Q^{(2)}_{6,l} F_3) \\ \label{errorx1}
 + &\sum_{l\in \Z}   (    P^{(1)}_{1,l+51} F_1)\, ( P^{(1)}_{2,l-1}\mathcal{Q}^{(2)}_{l}F_2)\, ( Q^{(2)}_{6,l} F_3) ,
\end{align}
where  the symbol of  $\mathcal{Q}_l$  is up to a constant  adapted in $[-2^{l+51},2^{l+51}]$ and it vanishes near the origin.  The term \eqref{errorx1} is bounded by H\"older's inequality and bounds on the maximal and square functions in the full range. 
For  \eqref{errorx}
one can proceed analogously as for  \eqref{tel:boundary}, with the roles of the first and second fiber interchanged.
This gives bounds 
in the range  $1<p_1, p_2,p_3<\infty$, $2<p_4<\infty$ and H\"older scaling.
\end{proof}

Now we are ready to prove Lemma \ref{mainprop-tri}.
\begin{proof}
[Proof of Lemma \ref{mainprop-tri}]
 For simplicity of notation, we shall   assume that $P_{1,k}=P_{2,k}=P_k$ and $Q_{1,k}=Q_{2,k}=Q_k$, but as it will be clear from the proof, the arguments will work for general Fourier multiplier operators with symbols satisfying the above assumptions.  

The proof  will be based on analyzing the ordering of the parameters $k,l$, and $m$. Note that by symmetry we may assume $k\leq l\leq m$.
 If, in addition,   it holds that $k\sim l$ and $l\sim m$, then all frequencies are comparable up to an absolute constant and  the claim follows by H\"older's inequality.  

Consider now the case with two dominating frequencies when  $k\ll l$ and $l\sim m$, i.e. $l=m+s$ for some $-50\leq s\leq 0$. 
Then we need to estimate   
$$ \sum_{m\in \Z}\; \sum_{-50\leq s \leq 0}  \sum_{\substack{k\in \Z:\\  k\ll m+s}}\;(Q^{(1)}_{k} P^{(2)}_{m}  F_1)\,  (Q^{(1)}_{m+s}  P^{(2)}_{k} F_2)\, (Q^{(1)}_{m}  P^{(2)}_{m+s} F_3)  .$$
By Cauchy-Schwarz in $m$ and H\"older's inequality, its $\L^{p_4'}$ norm is bounded by 
\begin{align*}
   & \sum_{-50\leq s \leq  0}  \Big \|   \sum_{\substack{k\in \Z:\\  k\ll m+s}}  (Q^{(1)}_{k}P^{(2)}_{m}F_1) (Q^{(1)}_{m+s}P^{(2)}_{k}F_2) \,  \Big \|_{\L^{p'_3}(\ell^2_m)}  \|  Q^{(1)}_{m}P^{(2)}_{m+s}F_3  \|_{\L^{p_3}(\ell^2_m)}  
   \end{align*}
whenever   $1<p_3<\infty$   and $1/p_4'=1/p_3+1/p_3'$. For a fixed $s$, the second term is a square function, bounded in the full range.  For the first term we use Kintchine's inequality, which reduces to showing  
\begin{align*}
\Big \|  \sum_{|m|\leq M} a_m   \sum_{\substack{k\in \Z:\\  k\ll m+s}}  (Q^{(1)}_{k}P^{(2)}_{m}F_1) (Q^{(1)}_{m+s}P^{(2)}_{k}F_2) \Big \|_{\L^{p'_3}(\R^2)} \lesssim_{p_1,p_2} \|F_1\|_{\L^{p_1}(\R^2)}\|F_2\|_{\L^{p_2}(\R^2)}
\end{align*}
uniformly in $M>0$, where $|a_m|\leq 1$. 
This estimate holds when $1<p_1,p_2,p_3,p_4<\infty$, $1/p_1+1/p_2=1/p_3'$, $1/p_1+1/p_2>1/2$. Indeed, this follows from Theorem \ref{thm:doubletwist} after  normalizing the symbols of $P_m$ and $Q_{m+s}$.  
In the end, it remains to sum in $-50\leq s\leq 0$.   

Now we consider the $k\leq l\ll m$ with  one dominating frequency.  We dualize and consider the corresponding trilinear form.  
By the frequency localizations it suffices to bound 
\begin{align*}
\sum_{(k,l,m)\in \Z^3: k\leq l\ll m} \int_{\R^2} (Q^{(1)}_{k} P^{(2)}_{m}  F_1) \,  (Q^{(1)}_{l}  P^{(2)}_{k} F_2) \, (Q^{(1)}_{m}  P^{(2)}_{l} F_3) \, (\mathcal{Q}_{m}^{(1)} \mathcal{P}_{m}^{(2)}F_4 ),
\end{align*}
where $\mathcal{Q}_{m}$, $\mathcal{P}_{m}$ are adapted to $[-2^{m+6}, 2^{m+6}]$ and $\mathcal{Q}_m$ vanishes in $[-2^{m-6}, 2^{m-6}]$. As in the proof of Lemma \ref{lemma:tri-2}, we note that 
$\varphi_m(0)=c_0$ is a constant with $|c_0|\leq 1$ for each  $m\in \Z$.
 We write  $P_m = c_0I + (P_m -c_0I)$ for the projection acting on the second fiber of $F_1$. 
Then this reduces to  
\begin{align*}
\sum_{(k,l,m)\in \Z^3: k\leq l\ll m}   \mathcal{M}_{k,l,m} \quad +\quad  \sum_{(k,l,m)\in \Z^3: k\leq l\ll m}   \mathcal{E}_{k,l,m} ,
\end{align*}
where we have set
\begin{align}\label{main2}
  \mathcal{M}_{k,l,m} &= \int_{\R^2}  (Q^{(1)}_{k}    F_1) \,  (Q^{(1)}_{l}  P_k^{(2)} F_2) \, (Q^{(1)}_{m}  P^{(2)}_{l} F_3) \, (\mathcal{Q}_{m}^{(1)} \mathcal{P}_{m}^{(2)}F_4 ), \\\label{error2}
   \mathcal{E}_{k,l,m} &= \int_{\R^2}  (Q^{(1)}_{k} \widetilde{\mathcal{Q}}^{(2)}_{m}   F_1) \,  (Q^{(1)}_{l}    P_k^{(2)}F_2) \, (Q^{(1)}_{m}  P^{(2)}_{l} F_3) \, (\mathcal{Q}_{m}^{(1)} \mathcal{P}_{m}^{(2)}F_4 ).
\end{align}
Here  $\widetilde{\mathcal{Q}}_{m}$ is associated with a bump function, which is up to a constant multiple, adapted to  $[-2^{m+10},2^{m+10}]$. Moreover, the bump function vanishes at  the origin. We have redefined $\mathcal{P}_m$ by subsuming $c_0$ into its definition.

First we consider \eqref{main2}.  We write   the sum over $(k,l,m)\in \Z^3$ with $k \leq l \ll m$ as 
\begin{align}
\label{splitsum1}
\sum_{(k,l,m)\in \Z^3: k\leq l\ll m}   \mathcal{M}_{k,l,m}  &= \Big ( \sum_{(k,l,m)\in \Z^3: k\leq l\ll m}   \mathcal{M}_{k,l,m}  \, - \sum_{(k,l,m)\in \Z^3: k\leq l}    \mathcal{M}_{k,l,m}  \Big)\\
& \quad  +\quad  \sum_{(k,l,m)\in \Z^3: k\leq l}   \mathcal{M}_{k,l,m}.  \label{splitsum2} 
\end{align}
Let us first focus on    \eqref{splitsum2}.
We split the summation into the sums over $m$ and $(k,l)$, apply  Cauchy-Schwarz in $m$ and H\"older's inequality in the integration. This yields
\begin{align}\nonumber
&\Big | \sum_{(k,l,m)\in \Z^3:k\leq l}    \mathcal{M}_{k,l,m}  \Big |\\ &\leq  \Big\|  \sum_{(k,l)\in \Z^2:k\leq l}(Q^{(1)}_{k}    F_1) \,  (Q^{(1)}_{l}  P_k^{(2)} F_2) \, (Q^{(1)}_{m}  P^{(2)}_{l} F_3) \Big \|_{\L^{p'_4}(\ell^2_m)}  \| \mathcal{Q}_{m}^{(1)} \mathcal{P}_{m}^{(2)}F_4    \|_{\L^{p_4}(\ell^2_m)}.\label{main-unconstr}
\end{align}
 Since the second term is a square function, it remains  to prove $\L^{p_1}(\R^2) \times \L^{p_2}(\R^2) \times \L^{p_3}(\ell^2) \rightarrow \L^{p_4'}(\ell^2)$ vector-valued estimates for  the trilinear operator 
\begin{align}\label{form2-ter}
\sum_{(k,l) \in \Z^2:k\leq l} (Q^{(1)}_{k}    F_1) \,  (Q^{(1)}_{l}  P_k^{(2)} F_2) \, ( P^{(2)}_{l} F_3).     
\end{align}
These estimates will again follow by freezing the functions $F_1$ and $F_2$ and using   Marcinkiewicz-Zygmund inequalities, provided we show  scalar-valued boundedness of \eqref{form2-ter}. 
To show  this, we split the summation in \eqref{form2-ter} into regions where $k\sim l$, $k\ll l$, or $k\gg l$. 
The case   $k\sim l$    is bounded by H\"older's inequality. The remaining two cases  follow by Lemmas \ref{lemma:tri-1} and \ref{lemma:tri-2} above. Summarizing, we obtain   estimates for \eqref{main-unconstr} when  $1<p_1,p_2,p_3<\infty$, $2<p_4<\infty$,  $1/p_1+1/p_2 >1/2$, $1/p_2+1/p_3>1$, and $\sum_{i=1}^4 1/p_i =1$.

It remains to estimate \eqref{splitsum1}, which splits further into several terms. We fix $k\leq l$ and then we consider several cases depending on the size of $m$ relative to $k$ and $l$.  
If for any triple  $(j_1,j_2,j_3)\in \{k,l,m\}^3$ with all three distinct entries
it holds $j_1\sim j_2$ and $j_2\sim j_3$, then the claim follows by H\"older's inequality.  If $l\gg \max(k,m)$, then the corresponding term is zero due to frequency supports. 
Therefore, the remaining two cases with $k\leq l$ are:   $k\sim l$ and $m\ll k$, or  $k\ll l$ and $m\sim l$.

 Consider first the case when $k\sim l$ and $m\ll k$. Then we have two $Q$-type projections   in each parameter. Bounds are    deduced by two applications of H\"older's inequality as follows. Taking the triangle inquality, enlarging the sum in $m$ to be over the whole $\Z$ and applying  H\"older's inequality in $l$, we obtain 
\begin{align*}
& \Big| \sum_{-50\leq s\leq  0}  \sum_{l\in \Z} \sum_{m\in \Z: m\ll l+s} \mathcal{M}_{l+s,l,m} \Big|\\
& \leq \sum_{-50\leq s\leq  0} \int_{\R^2} \sum_{m\in \Z} \|Q^{(1)}_{l+s}F_1\|_{\ell^2_l}      \,  \|Q^{(1)}_{l}  P_{l+s}^{(2)} F_2\|_{\ell^2_l} \, \|Q^{(1)}_{m}  P^{(2)}_{l} F_3\|_{\ell^\infty_l} \, |\mathcal{Q}_{m}^{(1)} \mathcal{P}_{m}^{(2)}F_4 |
\end{align*}
By Cauchy-Schwarz in $m$   and H\"older in the integration,  we bound the term for a fixed $s$ by
\begin{align}\label{sqmax}
  \|Q^{(1)}_{l+s}F_1\|_{\L^{p_1}(\ell^2_l)}      \,  \|Q^{(1)}_{l}  P_{l+s}^{(2)} F_2\|_{\L^{p_2}(\ell^2_l)} \, \|Q^{(1)}_{m}  P^{(2)}_{l} F_3\|_{\L^{p_3}(\ell^2_m(\ell^\infty_l))} \, \|\mathcal{Q}_{m}^{(1)} \mathcal{P}_{m}^{(2)}F_4 \|_{\L^{p_4}(\ell^2_m)}.
\end{align}
Each of the terms is bounded in full range.
The third term,  square-maximal function term is bounded in by  the Fefferman-Stein inequality, see \cite{mptt:bi}. The remaining terms are square functions.

Next we consider the case $k\ll l$ and $m\sim l$. We split 
  $$\sum_{\substack{(k,l,m)\in \Z^3: \\  k\ll l, m\sim l}}   \mathcal{M}_{k,l,m} = \sum_{\substack{(k,l, m)\in \Z^3:\\ m\sim l}}    \mathcal{M}_{k,l,m}\quad - \sum_{\substack{(k,l, m)\in \Z^3:\\
l-50\leq k \leq l+100, m\sim l}}    \mathcal{M}_{k,l,m}  \quad - \sum_{\substack{(k,l, m)\in \Z^3: \\
 k\gg l+50,m\sim l}}    \mathcal{M}_{k,l,m}.  $$
Note that bounds for the second sum  follow by H\"older's inequality, while the third sum is zero due to frequency supports. Thus, it remains to consider the case $m\sim l$.  We again split the sum over $(k,l)\in \Z^2$ and write  $m=l+s$,              $s\sim 0$.
Now we use H\"older's inequality in $l$ to bound  
\begin{align}\nonumber
 & \sum_{s\sim 0} \Big |   \sum_{(k,l)\in \Z^2}   \mathcal{M}_{k,l,l+s}  \Big |\\
& \leq \sum_{s\sim 0}  \Big\|\sum_{k\in\Z} (Q^{(1)}_{k}   F_1) \,(Q_l^{(1)}P_{k}^{(2)} F_2) \Big     \|_{\L^{r}(\ell^2_l)} \|Q_{l+s}^{(1)}P_l^{(2)}F_3\|_{\L^{p_3}(\ell^2_l)}  \|Q_{l+s}^{(1)}P_{l+s}^{(2)}F_3\|_{\L^{p_4}(\ell^\infty_l)}\label{trilin-last}
\end{align}  
where ${1/r+1/p_3+1/p_4} =1$.  The second and third term are maximal and square functions, respectively, while 
bounds for the first  term    follow by  $\L^p(\R^2)\times \L^q(\ell^2) \rightarrow \L^r(\ell^2)$ vector-valued estimates for the operator \eqref{op:tp}.  Summing in $s$, we obtain estimates in the range     $1<p_1,p_2,p_3,p_4<\infty$,  $1/p_1+1/p_2 >1/2$, and  $\sum_{i=1}^4 1/p_i =1.$

To bound \eqref{error2}, we proceed analogously as for \eqref{main2}, except that  in the analogue of \eqref{form2-ter} we    use  $\L^{p_1}(\ell^2) \times \L^{p_2}(\R^2) \times \L^{p_3}(\ell^2) \rightarrow \L^{p_4}(\ell^2)$ vector-valued estimates for that operator and in the display  analogous to  \eqref{sqmax}   we now have the maximal-square function 
$$ \| Q^{(1)}_{l+s}  \widetilde{\mathcal{Q}}^{(2)}_{m}   F_1\|_{\L^{p_1}(\ell^\infty_m(\ell^2_l))},$$
which is also bounded in the full range by the Fefferman-Stein inequality.
The analogue for the vector-valued estimates used in \eqref{trilin-last} are now 
$\L^p(\ell^2)\times \L^q(\ell^2) \rightarrow \L^r(\ell^2)$ vector-valued estimates for the operator of the form \eqref{op:tp}.   
 \end{proof}

Finally, we are ready to tackle  the operators that result after the cone decomposition.  Recall that we need to prove bounds for $U_1$ and $U_2$ under the assumptions that for $1\leq i \leq 3$,
  $\psi_{i,k}$   is adapted to $[-2^{k+1},2^{k+1}]$, and for $4\leq j \leq 6$,
$\varphi_{j,k}$  is  adapted to $[-2^{k+1},2^{k+1}]$. Moreover, each $\psi_{i,k}$ vanishes on $[-2^{k-1},2^{k-1}]$ and $\varphi_{j,k}(0)=\varphi_{j,0}(0)$.
\subsection{Completing the proof of Theorem \ref{thm2}}
Let $\phi_{k}$ be a function adapted in $[-2^{k+1},2^{k+1}]$, which satisfies $\phi_{k}(0)=\phi_{0}(0)$ for each $k\in \Z$. 
We write 
\begin{align}
\label{addsub}
\phi_{k}=c_0^{-1} (c_0\phi_{k} - \widetilde{\phi}_{k}) + c_0^{-1}\widetilde{\phi}_{k},
\end{align}
where $\widetilde{\phi}_{k}$ is constantly equal to $c_0\phi_{0}(0)$ on  $[-2^{k-1},2^{k-1}]$, and  $|c_0|\leq 1$ is such that  $\widetilde{\phi}_{k}$ is a   function adapted in $[-2^{k+1},2^{k+1}]$.  Note that the function $c_0\phi_{k} - \widetilde{\phi}_{k}$ vanishes at the origin.

Consider $U_1$ and $U_2$ which resulted from the cone decomposition.
Writing each of the symbols of the $P$-type operators as in \eqref{addsub}, it suffices to bound   the operators under the following assumptions. For each $1\leq i\leq 3$,   $\psi_{i,k}$ is adapted in $[-2^{k+1},2^{k+1}]$ and vanishes on $[-2^{k-1},2^{k1}]$. 
For each $4\leq j \leq 5$,  $\varphi_{j,k}$ is supported in $[-2^{k+1}, 2^{k+1}]$  and exactly
one of the following holds: \begin{enumerate}
\item For each $4\leq j \leq 6$, the function $\varphi_{j,k}$ is constantly equal to $\varphi_{j,0}(0)$ on $[-2^{k-1},2^{k-1}]$.
\item There is an index   $4\leq j_0 \leq 6$, such that $\varphi_{j_0,k}(0)=0$ for each $k$. Moreover, for $j\neq j_0$, 
$\varphi_{j,k}$ is constantly equal to $\varphi_{j,0}(0)$ on $[-2^{k-1},2^{k-1}]$.
\item There is an index $4\leq j_0 \leq 6$, such that $\varphi_{j_0,k}$ is constantly equal to $\varphi_{j_0,0}(0)$ on $[-2^{k-1},2^{k-1}]$. Moreover,   for $j\neq j_0$, it holds $\varphi_{j,k}(0)=0$ for each $k$.
\item For each $4\leq j \leq 6$, it holds $\varphi_{j,k}(0)=0$ for each $k$.
\end{enumerate}

 We will analyze each of these cases for the operators of the form $U_1$ and $U_2$ and 
we start with $U_1$. Our first aim is to reduce considerations to the case when the symbols $\psi_{i,k}$ are supported in $[-2^{k+3}, 2^{k+3}]$, vanish on $[-2^{k-3}, 2^{k-3}]$, and are of the form $\varphi_{i,k}-\varphi_{i,k-1}$ for a function $\varphi_{i,k}$ supported in $[-2^{k+3},2^{k+3}]$.
The argument that follows is along the lines of an argument used in Section 6 of  \cite{vk:tp} when transitioning from the dyadic to the continuous setting.

Let us  denote by $P_{\varphi}$ the one-dimensional Fourier multiplier with symbol $\varphi$, i.e. $P_\varphi  f = f*\widecheck{\varphi}$. As before we shall denote its fiber-wise action on a two-dimensional function with a superscript. 
Let $\phi$ be a non-negative smooth function supported in $[-2^{-0.4},2^{-0.4}]$ and 
 such that ${\phi}$ is constantly equal to one on   $[-2^{-0.6},2^{-0.6}]$.  For $a\in \R$   define $\vartheta_a$ and $\rho_a$ by
\begin{align*}
& {\vartheta_a}(\xi) = {\phi}(2^{-a-1}\xi)-{\phi}(2^{-a}\xi),\\
& {\rho_a}(\xi) = {\phi}(2^{-a-0.6}\xi) - {\phi}(2^{-a-0.5}).
\end{align*}
Note that ${\vartheta_a}$ is supported in 
$[-2^{a+0.6},2^{a+0.6}]$  and vanishes on $[-2^{a-0.6},2^{a-0.6}]$.
Moreover, ${\vartheta_a}$ is constantly equal to one on $[2^{a-0.4},2^{a+0.4}]$ and $[-2^{a+0.4},-2^{a-0.4}]$.
Finally, $\rho_a$ is supported in 
$[-2^{a+0.2},2^{a+0.2}]$  and vanishes on $[-2^{a-0.1},2^{a-0.1}]$.
 In particular, we have ${\vartheta_a} =1$ on the support of  ${\rho_a}$. Moreover, for $k\in \Z$ and $1\leq i \leq 3$ we have 
\begin{align}\label{rho-prop}
&{\textstyle \sum_{l=-20}^{20}}\,  {\rho_{k+0.1l}} =1\quad \textup{on} \quad \textup{supp}({\psi_{i,k}}),\\\label{rho-prop1}
& \textstyle{\sum_{l=-20}^{20}}\, {\rho_{k+0.1l}} =0\quad \textup{on} \quad \textup{supp}({\psi_{i,k'}}) \quad \textup{if} \quad |k'-k|\geq 10.
\end{align}
 Let $n\in \Z$, $0\leq s \leq 9$. 
 Note that due to \eqref{rho-prop1} and \eqref{rho-prop} we may write for $\xi,\eta\in \R$
\begin{align*}
\sum_{k\in\Z}   \psi_{i,k}(\xi)\varphi_{j,k}(\eta) &=  \sum_{s=0}^9 \sum_{n\in \Z}   {\psi_{i,10n+s}}  (\xi) \varphi_{j,10n+s}(\eta)\\
&= \sum_{s=0}^9   \sum_{l=-20}^{20}\sum_{n\in \Z}  \rho_{10n+s+0.1l} (\xi) \Psi_{i,s}(\xi) \varphi_{j,10n+s}(\eta),
\end{align*} 
 where ${\Psi_{i,s}} = \sum_{k\in \Z}  \psi_{i,10k+s}$.  
 
Let us now consider $U_1$ with bump functions satisfying (Q) and (P). 
Applying the above considerations in each parameter $k,l,m$, it suffices to prove bounds for   
\begin{align*}
\sum_{(n_1,n_2,n_3)\in \Z^3}  \Big( ({P}^{(1)}_{\rho_{10n_1+s_1+0.1l_1}} {P}^{(2)}_{\varphi_{4,10n_3+s_3}}  G_{1,s_1})\, 
& ({P}^{(1)}_{\rho_{10n_2+s_2+0.1l_2}}  P^{(2)}_{\varphi_{5,10n_1+s_1}} {G}_{2,s_2})\\
& \qquad \qquad  ({P}^{(1)}_{\rho_{10n_3+s_3+0.1l_3}} P^{(2)}_{\varphi_{6,10n_2+s_2}} {G}_{3,s_3}) \Big )
\end{align*}
for fixed $0\leq s_1,s_2,s_3\leq 9$ and $-20\leq l_1,l_2,l_3\leq 20$, where $G_{i,s} = P^{(1)}_{\Psi_{i,s}} {F}_i$. The Mikhlin-H\"ormander theorem in one variable gives
$$\|G_{i,s}\|_{\L^{p_i}(\R^2)}  = \|P^{(1)}_{\Psi_{i,s}}{F}_i\|_{\L^{p_i}(\R^2)} \lesssim_s \|{F}_i\|_{\L^{p_i}(\R^2)}.$$
Now, for  $b\in \R$ and $0\leq s\leq 9$  define
${G}_{i,s,b} = \sum_{n\in \Z} P^{(1)}_{\rho_{10n+s+b}}   G_{i,s}.$ 
By   support considerations,
$$P^{(1)}_{\vartheta_{k+b}}{G}_{i,s,b}
 = P^{(1)}_{\rho_{10n+s+b}}{G}_{i,s} \quad \textup{if}\quad k=10n+s\in 10\Z+s$$ and  $P^{(1)}_{\vartheta_{k+b}}{G}_{i,s,b}=0$   if $k\not\in 10n+s$.   
 The Littlewood-Paley inequality in one variable gives
\begin{align*}
\|G_{i,s,b}\|_{\L^{p_i}(\R^2)}\lesssim_{p_i,s,b} \|G_{i,s}\|_{\L^{p_i}(\R^2)}.
\end{align*}
Thus, it suffices to bound
\begin{align}\label{U1-finalreduction}
\sum_{(k_1,k_2,k_3)\in \Z^3}  ({P}^{(1)}_{\vartheta_{k_1+b_1}} {P}^{(2)}_{\varphi_{4, k_3}}  G_{1,s_1,b_1})\,  ({P}^{(1)}_{\vartheta_{k_2+b_2}}  P^{(2)}_{\varphi_{5,k_1}} G_{2,s_2,b_2})\, ({P}^{(1)}_{\vartheta_{k_3+b_3}} P^{(2)}_{\varphi_{6,k_2}} G_{3,s_3,b_3})
\end{align}
for each fixed $b_i=0.1l_i$,   $-20\leq l_i\leq 20$ and $0\leq s_i\leq 9$, $1\leq i\leq 3$.

Note that $\vartheta_b$ is supported in $[-2^3,2^3]$ and vanishes on $[-2^{-3},2^{-3}]$. Moreover, $(\vartheta_b)_k$ equals $(\phi_{b})_k - (\phi_{b})_{k-1}$, where $\phi_b=\phi(2^{-b-1}\cdot)$ is supported in  $[-2^3,2^3]$.

Now we are ready to distinguish further cases depending on the form of $\varphi_{j,k}$.

{\bf Case (1) of $U_1$.} 
Bounds for \eqref{U1-finalreduction}   follow from  Lemma~\ref{mainprop-tri}, applied with ${\psi}_{i,k}$ being a constant multiple of $\vartheta_{k+b_i}$ for $1\leq i \leq 3$.

{\bf Case (2) of $U_1$.} Without loss of generality, we may assume $j_0=4$. 
Then we dualize and write the quadrilinear form in question up to a constant as 
\begin{align*}
\int_{\R^2} \sum_{(k,m)\in \Z^2} (Q^{(1)}_{1,k} P^{(2)}_{4,m}  F_1)(x) \mathcal{Q}^{(1)}_{1,k} \Big( \sum_{l\in \Z} 
 (Q^{(1)}_{2,l}  P^{(2)}_{5,k} F_2)(x)\, (Q^{(1)}_{3,m}  P^{(2)}_{6,l} F_3)(x)  F_4(x)\Big)  dx,
\end{align*}
where $\mathcal{Q}_{1,k}$ is a constant multiple of a function adapted in $[-2^{4},2^{4}]$, which   vanishes on $[-2^{-4},2^{-4}]$.  By Cauchy-Schwarz in $k,m$, and H\"older's inequality in the integration,  it suffices to bound a fiber-wise square function and 
\begin{align*}
\Big \|\Big( \sum_{k\in \Z}\Big| \mathcal{Q}^{(1)}_{1,k} \sum_{l\in \Z}  
 (Q^{(1)}_{2,l}  P^{(2)}_{5,k} F_2)(x)\, (Q^{(1)}_{3,m}  P^{(2)}_{6,l} F_3)(x)  F_4(x)\Big|^2\Big)^{1/2} \Big \|_{\L^{p_1'}(\ell^2_m)}.
\end{align*}
Dualizing with a function $\widetilde{F}_1$ and using Kintchine's inequality we see that it suffices to prove $\L^{p_1}\times \L^{p_2}\times \L^{p_3}\rightarrow \L^{p_4'}$ bounds for 
the operator
\begin{align*}
 \sum_{(k,l)\in \Z}a_k (\mathcal{Q}^{(1)}_{1,k} \widetilde{F}_1)(x)
 (Q^{(1)}_{2,l}  P^{(2)}_{5,k} F_2)(x)\, ( P^{(2)}_{6,l} F_3)(x) .
\end{align*}
Now we perform the analogous averaging argument as the one that led to  \eqref{U1-finalreduction}. Then it suffices to consider the case when $a_k=1$ and $\mathcal{Q}_{1,k}$ has a symbol   supported in $[-2^{9}, 2^{9}]$, which vanishes on $[-2^{-9}, 2^{-9}]$, and   is of the form $\theta_k - \theta_{k-1}$ for a suitable function $\theta$. The desired bounds then  follow by Lemmas \ref{lemma:tri-1} and \ref{lemma:tri-2}.

{\bf Case (3) of $U_1$.} Without loss of generality,  we may  assume $j_0=6$.  Then we  write the operator in question as 
\begin{align*}
\sum_{(k,m)\in \Z^2} (Q^{(1)}_{1,k} P^{(2)}_{4,m}  F_1)(x) \Big( \sum_{l\in \Z} 
 (Q^{(1)}_{2,l}  P^{(2)}_{5,k} F_2)(x)\, (Q^{(1)}_{3,m}  P^{(2)}_{6,l} F_3)(x) \Big).
\end{align*}
By the Cauchy-Schwarz inequality in $k$ and $m$, this reduces to vector-valued estimates for the operator \eqref{op:tp} and a fiber-wise square function. We obtain $\L^{p_1}\times \L^{p_2}\times \L^{p_3}\rightarrow \L^{p_4'}$ estimates with H\"older scaling and in the range $1<p_1,p_2,p_3,p_4<\infty$, $1/p_2+1/p_3>1/2$.

{\bf Case (4) of $U_1$.} This case easily follows from three applications of H\"older's inequality, reducing to fiber-wise square functions. \\

It remains to treat $U_2$.  By the analogous argument as leading to \eqref{U1-finalreduction},   it suffices to consider the case when for $1\leq i \leq 3$, the symbols $\psi_{i,k}$ are supported in $[-2^{k+3}, 2^{k+3}]$, vanish on $[-2^{k-3}, 2^{k-3}]$, and are of the form $\varphi_{i,k}-\varphi_{i,k-1}$ for a function $\varphi_{i,k}$ supported in $[-2^{k+3},2^{k+3}]$. Then we distinguish further cases depending on the form of $\varphi_{j,k}$.

{\bf Case (1) of $U_2$.} 
Let $\widetilde{Q}_{5,k}$ and $\widetilde{P}_{1,k}$ be  associated with $\widetilde{\psi}_{5,k}$ and $\widetilde{\varphi}_{1,k}$, respectively, where  
$$\widetilde{\psi}_{5,k}=\varphi_{5,k}-\varphi_{5,k-1}$$  and 
$\widetilde{\varphi}_{1,k}=\varphi_{1,k-1}$ (with $\varphi_{1,k}$ defined in \eqref{psiphicond}). Observe that $\widetilde{\psi}_{5,k}$ is supported in $[-2^{k+3},2^{k+3}]$ and vanishes on $[-2^{k-3},2^{k-3}]$. Moreover,  $\widetilde{\varphi}_{1,k}$ is supported in $[-2^{k+4},2^{k+4}]$ and observe  that $\widetilde{\varphi}_{1,k}-\widetilde{\varphi}_{1,k-1} = \varphi_{1,k-1}-\varphi_{1,k-2}=\psi_{1,k-1}$ is supported in $[-2^{k+4},2^{k+4}]$ and vanishes  on $[-2^{k-4},2^{k-4}]$. 
An application of the telescoping identity    \eqref{telescoping}   in $k$ yields that $U_2$ then equals   
 \begin{align} \label{tel1}
 & - \sum_{(k,l,m)\in \Z^3} (\widetilde{Q}^{(1)}_{5,k} P^{(2)}_{4,m}  F_1)\,  (Q^{(1)}_{2,l}  \widetilde{P}^{(2)}_{1,k} F_2)\, (Q^{(1)}_{3,m}  P^{(2)}_{6,l} F_3) \\\label{tel2}
&+   \sum_{(l,m)\in \Z^2} (P^{(2)}_{4,m}  F_1)\,  (Q^{(1)}_{2,l}   F_2)\, (Q^{(1)}_{3,m}  P^{(2)}_{6,l} F_3)
\end{align}
Bounds for the first term follow from Lemma \ref{mainprop-tri}, applied with $\psi_{1,k}$ being $\widetilde{\psi}_{5,k}$ and $\varphi_{5,k} = \widetilde{\varphi}_{1,k}$. 
The desired bounds for the second term   follow from  Lemmas \ref{lemma:tri-1} and   \ref{lemma:tri-2}, and by H\"older's inequality used for the portion of the sum when $l\sim m$.

{\bf Case (2) of $U_2$.} 
Let $j_0=6$. Performing the analogous steps as in Case (2) of $U_1$, we see that it suffices to show    estimates for 
\begin{align*}
\sum_{(k,m)\in \Z^2} a_m ({P}^{(1)}_{5,k} P^{(2)}_{4,m}  F_1)\,  ({Q}^{(2)}_{1,k}    {F}_2)\, (\mathcal{Q}^{(1)}_{3,m}    \widetilde{F}_3) ,
\end{align*}
where $\mathcal{Q}_{3,m}  $ is as in Case (2) of $U_1$. 
By an averaging argument (as leading to \eqref{U1-finalreduction}) it suffices to consider the case when  $a_k=1$ and $\mathcal{Q}_{3,m}$ has a symbol   supported in $[-2^{9}, 2^{9}]$, which vanishes on $[-2^{-9}, 2^{-9}]$, and  is of the form $\theta_k - \theta_{k-1}$ for a suitable function $\theta$.  The clam then follows by the telescoping identity \eqref{telescoping} in $m$,  Lemmas \ref{lemma:tri-1} and \ref{lemma:tri-2}, and bounds for \eqref{op:tp}.
If $j_0=5$, we can  proceed as in Case (2) of $U_1$. 
Then we need to show estimates for
\begin{align*}
\sum_{(l,m)\in \Z^2} a_l ( P^{(2)}_{4,m}  F_1)\,  (\mathcal{Q}^{(1)}_{2,l}    \widetilde{F}_2)\, (Q^{(1)}_{3,m}   P_{6,l}^{(2)} F_3) .
\end{align*}
  If $j_0=4$,  we  first use the telescoping identity \eqref{telescoping} in $k$, giving   terms of the form \eqref{tel1} and \eqref{tel2}. 
For   \eqref{tel2}, we apply the Cauchy-Schwarz inequality in $m$, which leads to  vector-valued estimates for the operator \eqref{op:tp}.
For \eqref{tel1}, we can  proceed   as in Case (2) of $U_1$, up to obvious modifications.

{\bf Cases (3) and (4) of $U_2$.} 
These two cases are   analogous to  Cases (3) and (4) of $U_1$.  \qed

\end{document}